\documentclass[envcountsame,envcountsect]{svmult}
\usepackage{amsmath}
\usepackage{amscd,amssymb}
\usepackage[all]{xy}

\spnewtheorem{rem}[theorem]{Remark}{\bfseries}{\upshape}
\spnewtheorem{exm}[theorem]{Example}{\bfseries}{\upshape}
\spnewtheorem{quest}[theorem]{Question}{\bfseries}{\upshape}
\spnewtheorem{prb}[theorem]{Problem}{\bfseries}{\upshape}

\newcommand{\Z}{\mathbb{Z}}
\newcommand{\Q}{\mathbb{Q}}
\newcommand{\R}{\mathbb{R}}
\newcommand{\C}{\mathbb{C}}

\newcommand{\PP}{\mathbb{P}}
\newcommand{\CP}{\mathbb{CP}}

\newcommand{\RR}{{\mathcal R}}
\newcommand{\VV}{\mathcal{V}}

\newcommand{\wR}{\widetilde{\RR}}
\newcommand{\wV}{\widetilde{\VV}}

\newcommand{\A}{{\mathcal{A}}}
\newcommand{\B}{{\mathcal{B}}}

\newcommand{\m}{{\mathfrak{m}}}

\DeclareMathOperator{\init}{in}
\DeclareMathOperator{\rank}{rank}

\DeclareMathOperator{\im}{im}
\DeclareMathOperator{\coker}{coker}
\DeclareMathOperator{\codim}{codim}

\DeclareMathOperator{\ab}{{ab}}
\DeclareMathOperator{\Sym}{Sym}

\DeclareMathOperator{\Hom}{{Hom}}

\DeclareMathOperator{\ev}{ev}

\DeclareMathOperator{\corank}{corank}

\DeclareMathOperator{\TC}{TC}
\DeclareMathOperator{\supp}{supp}

\DeclareMathOperator{\Char}{Char}
\DeclareMathOperator{\rk}{rk}

\DeclareMathOperator{\Conf}{Conf}

\newcommand{\cdga}{\textsc{cdga}}
\newcommand{\surj}{\twoheadrightarrow}
\newcommand{\inj}{\hookrightarrow}

\newcommand{\isom}{\xrightarrow{\,\simeq\,}}

\newcommand{\abs}[1]{\left| #1 \right|}

\def\set#1{{\left\{#1\right\}}}
\newcommand{\apl}{A_{\rm PL}}
\def\dot{\mathchar"013A}  
\newcommand{\hdot}{{\raise1pt\hbox to0.35em{\Large $\dot$\!}}} 
\newcommand{\hsp}{\qquad\qquad\qquad}

\begin{document}

\title*{Around the tangent cone theorem}
\titlerunning{Around the tangent cone theorem}

\author{Alexander~I.~Suciu}
\institute{Alexander~I.~Suciu \at 
Department of Mathematics,
Northeastern University,
Boston, MA 02115, USA\\ 
\email{a.suciu@neu.edu}\\
Supported in part by National Security 
Agency grant H98230-13-1-0225.
}

\setcounter{minitocdepth}{1}
\maketitle
\dominitoc

\abstract{A cornerstone of the theory of cohomology 
jump loci is the Tangent Cone theorem, which relates the 
behavior around the origin of the characteristic and resonance 
varieties of a space.  We revisit this theorem, in both 
the algebraic setting provided by $\cdga$ models, and in 
the topological setting provided by fundamental groups 
and cohomology rings. The general theory is illustrated 
with several classes of examples from geometry and 
topology:  smooth quasi-projective varieties, complex 
hyperplane  arrangements and their Milnor fibers, configuration 
spaces, and elliptic arrangements. 
}


\section{Introduction}
\label{sect:intro}

The Tangent Cone theorem relates two seemingly disparate 
sets of cohomology jump loci associated to a space $X$:  
the resonance varieties, which are constructed from 
information encoded in either the cohomology ring of $X$, 
or an algebraic model for this space, and the characteristic 
varieties, which depend on the {\it a priori}\/ much more subtle 
information carried by the cohomology of $X$ with coefficients 
in rank $1$ local systems.  We focus here on the interplay 
between these two sets of jump loci, which are even more 
tightly related under certain algebraic (positivity of weights),  
topological (formality), or geometric (quasi-projectivity) 
assumptions. 

\subsection{Resonance varieties}
\label{subsec:res intro}

We start in \S\ref{sect:resonance} with a description of the 
various resonance varieties associated to a commutative,  
differential graded algebra (for short, a $\cdga$). 
We continue in \S\ref{sect:models} with the resonance 
varieties associated to a space $X$, using as input 
either its cohomology algebra or a suitable 
algebraic model, and discuss the 
algebraic version of the Tangent Cone theorem.

We will assume throughout that $X$ is a reasonably nice 
space, to wit, a connected CW-complex with finitely many cells 
in each dimension.  To such a space, we associate two types 
of resonance varieties.   The classical ones are obtained 
from the cohomology algebra $A=H^*(X,\C)$, by setting 
\begin{equation}
\label{eq:intro res}
\RR^i(X)=\{a \in A^1 \mid  H^i(A, \delta_a) \ne 0\},  
\end{equation}
where, for each $a\in A^1$, we denote by $(A,\delta_a)$ 
the cochain complex with differentials $\delta_a\colon A^i\to A^{i+1}$ 
given by left-multiplication by $a$.  Each set $\RR^i(X)$ is a 
homogeneous subvariety of the complex affine space $A^1=H^1(X,\C)$. 

Lately, an alternate definition of resonance has emerged 
(in works such as \cite{DP-ccm, DPS-duke, DPS-14, MPPS, PS-springer}), 
whereby one replaces the cohomology algebra by an algebraic 
model for $X$, that is, a commutative differential graded algebra 
$(A,\D)$ weakly equivalent to the Sullivan model of polynomial 
forms on $X$, as defined in \cite{Su77}. We may then 
set up a cochain complex $(A,\delta_a)$ as above, but 
now with differentials given by $\delta_a(u)=au+\D{u}$, 
and define the resonance varieties $\RR^i(A)\subset H^1(A)$ 
just as before. 

Assuming now that each graded piece $A^i$ is finite-dimensional, 
the sets $\RR^i(A)$ are subvarieties of the affine space $H^1(A)$, 
which depend only on the isomorphism type of $A$. These 
varieties are not necessarily homogeneous; nevertheless, as shown 
in \cite{MPPS}, the following inclusion holds,
\begin{equation}
\label{eq:tc0 intro}
\TC_0(\RR^i(A))\subseteq \RR^i(X),
\end{equation}
where $\TC_0$ denotes the tangent cone at $0\in H^1(A)$.

Under some additional hypothesis, one can say more. Suppose  
our finite-type model $(A,\D)$ admits a $\Q$-structure compatible 
with that of the Sullivan model, 
and also has positive weights, in the sense of \cite{Su77, Mo}. 
Then, $\RR^i(A)$ is a finite union of rationally defined 
linear subspaces of $H^1(A)$, and
\begin{equation}
\label{eq:rinc intro}
\RR^i(A) \subseteq \RR^i(X).
\end{equation}

\subsection{Characteristic varieties}
\label{subsec:cvar intro}

We turn in \S\ref{sect:charvar} to the characteristic varieties 
of a space $X$, and to the two types of tangent cones  
associated to them. This sets the stage for the topological 
version of the Tangent Cone theorem, which is treated 
in \S\ref{sect:tcone}. 

Unlike the resonance varieties, which arise from an 
algebraic model, the characteristic varieties arise from the 
chain complex of the universal abelian cover of the space. 
Let  $\pi=\pi_1(X)$ be the fundamental group of $X$, 
let $\pi_{\ab}=H_1(X,\Z)$ be its abelianization, and let 
$\Char(X)=\Hom(\pi_{\ab},\C^{*})$ be its group of complex-valued 
characters. Then 
\begin{equation}
\label{eq:intro cv}
\VV^i(X)=\{\rho \in \Char(X) \mid 
H_i(X, \C_{\rho})\ne 0\}, 
\end{equation}
where $\C_{\rho}$ denotes the complex vector space $\C$, 
viewed as a module over the group algebra $\C [\pi_{\ab}]$ via 
$g\cdot z = \rho(g) z$, for $g\in \pi$ and $z\in \C$. 

The relationship between the characteristic and resonance 
varieties of a space goes through the tangent cone construction.  
Let us start by identifying the tangent space at the identity to the 
complex algebraic group $\Char(X)$ with the complex affine 
space $H^1(X,\C)$.  Then, as shown in \cite{Li02, DPS-duke}, 
we have the following chain of inclusions:
\begin{equation}
\label{eq:tc inc intro}
\tau_{1}(\VV^i(X))\subseteq  \TC_{1}(\VV^i(X))\subseteq \RR^i(X), 
\end{equation}
where $\tau_{1}$ denotes the `exponential tangent cone' at the 
identity $1\in \Char(X)$ (a finite union of rationally defined linear subspaces), 
and $\TC_{1}$ denotes the usual tangent cone at $1$ (a homogeneous 
subvariety).

The crucial property that bridges the gap between the two types of 
tangent cones to a characteristic variety and the corresponding 
resonance variety is that of  formality, in the sense of  
Sullivan \cite{Su77}.  Given a $1$-formal space, 
one of the main results from \cite{DPS-duke} establishes 
an isomorphism between the analytic germ of $\VV^1(X)$ 
at $1$ and the analytic germ of $\RR^1(X)$ at $0$.

More generally, if $X$ has an algebraic model $A$ with 
good finiteness properties, then, as shown in \cite{DP-ccm}, 
the characteristic varieties $\VV^i(X)$ may be identified 
around the identity with the resonance varieties $\RR^i(A)$.
Consequently, if $X$ is formal (that is, the cohomology algebra 
of $X$, endowed with the zero differential, is weakly equivalent 
to the Sullivan model), then the following `Tangent Cone formula' holds:
\begin{equation}
\label{eq:tcone formula intro}
\tau_{1}(\VV^i(X)) =  \TC_{1}(\VV^i(X)) =  \RR^i(X).
\end{equation}

Consequently, if either one of the two inclusions in \eqref{eq:tc inc intro} 
fails to be an equality, the space $X$ is not formal. Viewed this way, the 
Tangent Cone theorem can be thought of as a (quite powerful) formality 
obstruction. 

\subsection{Quasi-projective varieties}
\label{subsec:qp intro}

We conclude our overview of cohomology jump loci with an 
exploration of the Tangent Cone theorem  in the framework of 
complex algebraic geometry.  We start in \S\ref{sect:qproj} 
with the general theory of jump loci of smooth, 
quasi-projective varieties. We then specialize in \S\ref{sect:arr mf} 
to complements of hyperplane arrangements and their Milnor 
fibers, and in \S\ref{sect:elliptic} to complements of 
elliptic arrangements.

Let $X$ be a smooth, complex quasi-projective variety.   
Work of Arapura \cite{Ar}, as recently  sharpened by 
Budur and Wang \cite{BW1}, reveals a profound fact 
about the characteristic varieties $\VV^i(X)$: they are 
all finite unions of torsion-translated subtori of 
the character group $\Char(X)$. 

Every quasi-projective variety as above can be realized as the 
complement, $X=\overline{X}\setminus D$, of a normal-crossings divisor $D$ 
in a smooth, complex projective variety $\overline{X}$. Given such a `good' 
compactification, Morgan associates in \cite{Mo} an algebraic model 
for our variety, $A(X)= A (\overline{X},D)$. This `Gysin' model 
is a finite-dimensional, rationally defined $\cdga$ with positive weights, 
which is weakly equivalent to Sullivan's model for $X$. 

Using the aforementioned work of Arapura and Budur--Wang, 
as well as work of Dimca--Papadima \cite{DP-ccm}, 
we obtain the following formulation of the Tangent Cone theorem 
for smooth, quasi-projective varieties $X$:
\begin{equation}
\label{eq:tcone qp intro}
\tau_{1}(\VV^i(X)) =  \TC_{1}(\VV^i(X)) = \RR^i(A(X)) \subseteq  \RR^i(X).
\end{equation}

In degree $i=1$, the irreducible components of $\VV^1(X)$ which 
pass through the identity are in one-to-one 
correspondence with the set $\mathcal{E}_X$ of `admissible' maps 
$f\colon X\to \varSigma$, where $\varSigma$ is a smooth complex 
curve with $\chi(\varSigma)<0$. This leads to a concrete description 
of the variety $\RR^1(A(X))$, and of the variety $\RR^1(X)$ when 
$X$ is $1$-formal. 

Especially interesting is the case when $X=M(\A)$ is the complement 
of an arrangement $\A$ of hyperplanes in some complex vector space. 
The cohomology algebra $A=H^*(X,\C)$ admits a combinatorial 
description, in terms of the intersection lattice of $\A$. Moreover, 
the $\cdga$ $(A, 0)$ is a model for $X$; thus, formula \eqref{eq:tcone qp intro} 
holds with equalities throughout. 

For an arrangement complement as above, work of Falk and Yuzvinsky \cite{FY} 
identifies the set $\mathcal{E}_{X}$ with the set of multinets on 
sub-arrangements of $\A$, up to relabeling (see also \cite{PS-beta}).   
This yields a completely combinatorial description 
of the resonance variety $\RR^1(X)$, and of the components 
of  the characteristic variety $\VV^1(X)$ passing through the identity. 

Another smooth variety associated to an arrangement $\A$ 
is the Milnor fiber $F=F(\A)$, defined as the level set $Q=1$, where 
$Q$ is a defining polynomial for $\A$. The topology of this variety 
(even its first Betti number!)  is much less understood.  
As shown by Zuber \cite{Zu}, though, the inclusion 
$\TC_1(\VV^1(F))\subset \RR^1(F)$ can be strict;  
hence, $F$ can be non-formal. Further understanding 
of how the Tangent Cone formula works in this context hinges 
on finding a good compactification for $F$, and then computing 
the corresponding Gysin model and its resonance varieties.

The machinery of cohomology jump loci can also be brought to 
bear in the study of elliptic arrangements. Let $E^{\times n}$ be the 
$n$-fold product  of an elliptic curve $E$.  An elliptic arrangement in 
$E^{\times n}$ is a finite collection of fibers of group homomorphisms 
$E^{\times n}\to E$.  Assuming that all subspaces in the intersection 
poset of $\A$ are connected, Bibby constructs in \cite{Bi} a finite-dimensional, 
algebraic model for the complement, which can be thought of as a 
concrete version of the Gysin model.

A special case of this construction is the configuration 
space $\Conf(E,n)$ of $n$ distinct, ordered points on 
$E$, itself a classifying space for 
the $n$-stranded pure braid group on the torus.  
We illustrate the general theory in a simple, 
yet instructive example.  Direct computation shows 
that, for $X=\Conf(E,3)$, the 
resonance variety $\RR^1(A(X))$ is properly contained  
in $\RR^1(X)$, thereby establishing the non-formality 
of $X$.

\section{The resonance varieties of a $\cdga$}
\label{sect:resonance}

We start with the resonance varieties associated to a 
commutative differential graded algebra, some of their 
properties, and various ways to compute them. 

\subsection{Commutative differential graded algebras}
\label{subsec:cdga}

Let $A=(A^{\hdot},\D)$ be a commutative, differential 
graded algebra over the field $\C$.  
That is, $A=\bigoplus_{i\ge 0} A^i$ is a graded $\C$-vector 
space, endowed with a multiplication map 
$\cdot\colon A^i \otimes A^j \to A^{i+j}$ satisfying 
$u\cdot v = (-1)^{ij} v \cdot u$,  
and a differential $\D\colon A^i\to A^{i+1}$ 
satisfying $\D(u\cdot v) = \D{u}\cdot v 
+(-1)^{i} u \cdot \D{v}$, for all $u\in A^i$ and $v\in A^j$. 

Unless otherwise stated, we will assume throughout that 
$A$ is connected, i.e., $A^0=\C$, and of finite-type, i.e., 
$A^i$ is finite-dimensional, for all $i\ge 0$. 

Using only the underlying cochain complex structure of the $\cdga$, 
we let $Z^i(A)=\ker (\D\colon A^i\to A^{i+1})$ and 
$B^i(A)=\im (\D\colon A^{i-1}\to A^{i})$, and set 
$H^i(A)=Z^i(A)/B^i(A)$.  The direct sum of the 
cohomology groups, $H^{\hdot}(A)=\bigoplus_{i\ge 0} H^i(A)$, 
inherits an algebra structure from $A$.

A morphism between two $\cdga$s, $\varphi\colon A\to B$, is both 
an algebra map and a cochain map. Consequently, $\varphi$ induces a 
morphism $\varphi^*\colon H^{\hdot} (A)\to H^{\hdot} (B)$ 
between the respective cohomology algebras.  
We say that $\varphi$ is a quasi-isomorphism if $\varphi^*$ is an 
isomorphism. Likewise, we say $\varphi$ is a $q$-isomorphism (for some 
$q\ge 1$) if $\varphi^*$ is an isomorphism in degrees up to $q$ 
and a monomorphism in degree $q+1$.  

Two $\cdga$s $A$ and $B$ are {\em weakly equivalent} 
(or just {\em $q$-equivalent}) if there is a zig-zag of quasi-isomorphisms 
(or $q$-isomorphisms) connecting $A$ to $B$, in which 
case we write $A\simeq B$ (or $A\simeq_q B$).  

A $\cdga$ $(A,\D)$ is said to be {\em formal}\/ (or just {\em $q$-formal}) 
if it is weakly equivalent (or just $q$-equivalent) to its cohomology 
algebra, $H^{\hdot}(A)$, endowed with the zero differential. 

Finally, we say that $(A,\D)$ is rationally defined if  
$A$ is the complexification of a graded $\Q$-algebra $A_{\Q}$, 
and the differential $\D$ preserves $A_{\Q}$.

We will also consider the dual vector spaces 
$A_i=(A^i)^{\vee}:=\Hom(A^i, \C)$, 
and the chain complex $(A_{\hdot},\partial)$, where 
$\partial \colon A_{i+1} \to A_i$ is the dual to $\D\colon A^i\to A^{i+1}$. 
If $H_i(A)$ are the homology groups of this chain complex, then,  
by the Universal Coefficients theorem, $H_i(A)\cong (H^i(A))^{\vee}$.

\subsection{Resonance varieties}
\label{subsec:res}
Our connectivity assumption on the $\cdga$ $(A,\D)$ 
allows us to identify the vector space $H^1(A)$ 
with the cocycle space $Z^1(A)$. For each element 
$a\in Z^1(A)\cong H^1(A)$, we turn $A$ into a cochain complex, 
\begin{equation}
\label{eq:aomoto}
\xymatrix{(A^{\hdot} , \delta_{a})\colon  \ 
A^0  \ar^(.65){\delta^0_{a}}[r] & A^1
\ar^(.5){\delta^1_{a}}[r] 
& A^2   \ar^(.5){\delta^2_{a}}[r]& \cdots },
\end{equation}
with differentials given by $\delta^i_{a} (u)= a \cdot u + \D{u}$, 
for all $u \in A^i$.  The cochain condition is verified 
as follows:  $\delta_a^{i+1}\delta^i_{a} (u)= a^2 u +a \cdot\D{u}  
+ \D{a}\cdot u -a \cdot\D{u} + \D\D u=0$.

Computing the homology of these chain complexes 
for various values of the parameter $a$, and 
keeping track of the resulting Betti numbers singles   
out certain {\em resonance varieties}\/ inside the 
affine space $H^1(A)$.
More precisely, for each non-negative integer $i$, define  
\begin{equation}
\label{eq:rra}
\RR^i(A)= \{a \in H^1(A)   
\mid  H^i(A^{\hdot}, \delta_{a}) \ne 0\}.
\end{equation}

These sets can be defined for any connected $\cdga$.  
If $A$ is of finite-type (as we always assume), the sets $\RR^i(A)$ are,  
in fact, algebraic subsets of the ambient affine space $H^1(A)$.  
Clearly, $H^i(A^{\hdot}, \delta_{0})=H^i(A)$; thus, the point 
$0\in H^1(A)$ belongs to the variety $\RR^i(A)$ 
if and only if $H^i(A)\ne 0$.  Moreover, $\RR^0(A)=\{0\}$.  

When the differential of $A$ is zero, the resonance varieties 
$\RR^i(A)$ are homogeneous subsets of $H^1(A)=A^1$. 
In general, though, the resonance varieties of a $\cdga$ 
are not homogeneous: see \cite[Example 2.7]{MPPS} and 
Example \ref{ex:nonhomog} below. 

The following lemma follows quickly from the definitions 
(see \cite[Lemma 2.6]{MPPS} for details). 

\begin{lemma}[\cite{MPPS}] 
\label{lem:functoriality}
Let $\varphi\colon A\to A'$ be a $\cdga$ morphism, 
and assume $\varphi$ is an isomorphism up to degree $q$, 
and a monomorphism in degree $q+1$, for some $q\ge 0$.  
Then the induced isomorphism in cohomology, 
$\varphi^*\colon H^1(A')\to H^1(A)$, identifies 
$\RR^i(A)$ with $\RR^i(A')$ for each $i\le q$,
and sends $\RR^{q+1}(A)$ into $\RR^{q+1}(A')$. 
\end{lemma}

\begin{corollary}
\label{cor:invariance}
If $A$ and $A'$ are isomorphic $\cdga$s, then their 
resonance varieties are ambiently isomorphic.  
\end{corollary}

The conclusions of Lemma \ref{lem:functoriality} do not follow  
if we only assume that $\varphi\colon A\to A'$ is a $q$-isomorphism.  
This phenomenon is illustrated in \cite[Example 2.7]{MPPS} and also in 
Example \ref{ex:nonhomog} below. 

As shown in \cite{PS-plms, PS-springer}, the resonance varieties behave 
reasonably well under tensor products:
\begin{equation}
\label{eq:rprod}
\RR^i(A\otimes A') \subseteq \bigcup_{p+q=i} \RR^p(A)\times \RR^q(A').
\end{equation}
Moreover, if the differentials of both $A$ and $A'$ are zero, then 
equality is achieved in the above product formula. 

In a similar manner, we can define a homological version of resonance 
varieties, by considering the chain complexes $(A_{\hdot},\partial^{\alpha})$ 
with differentials $\partial_i^{\alpha} = (\delta^i_a)^{\vee}$ 
for $\alpha\in H_1(A)$ dual to $a\in H^1(A)$, and setting 
\begin{equation}
\label{eq:hom res}
\RR_i(A)= \{\alpha \in H_1(A)   
\mid  H_i(A_{\hdot}, \partial^{\alpha}) \ne 0\}.
\end{equation}

\begin{lemma}
\label{lem:resdual}
For each $i\ge 0$, the duality isomorphism $H^1(A)\cong H_1(A)$ 
identifies the resonance varieties $\RR^i(A)$ and $\RR_i(A)$.
\end{lemma}

\begin{proof}
\smartqed
By the Universal Coefficients theorem (over the field $\C$), 
we have that $H_i(A^{\hdot}, \delta_{a}) \cong H_i(A_{\hdot}, \partial^{\alpha})$. 
The claim follows.
\qed
\end{proof}

\subsection{A generalized Koszul complex}
\label{subsec:univ aomoto}

Let us fix now a basis $\{ e_1,\dots, e_n \}$ for the complex vector space 
$H^1(A)$, and let $\{ x_1,\dots, x_n \}$ be the Kronecker dual basis 
for the vector space $H_1(A)=(H^1(A))^{\vee}$.  In the sequel, 
we shall identify the symmetric algebra $\Sym(H_1(A))$ 
with the polynomial ring $S=\C[x_1,\dots, x_n]$,
and we shall view $S$ as the coordinate ring of 
the affine space $H^1(A)$.

Consider now the cochain complex of free $S$-modules, 
\begin{equation}
\label{eq:univ aomoto}
\xymatrixcolsep{22pt}
(A^{\hdot} \otimes S,\delta) \colon 
\xymatrix{
\cdots \ar[r] 
&A^{i}\otimes S \ar^(.45){\delta^{i}}[r] 
&A^{i+1} \otimes S \ar^(.5){\delta^{i+1}}[r] 
&A^{i+2} \otimes S \ar[r] 
& \cdots},
\end{equation}
where the differentials are the $S$-linear maps defined by 
\begin{equation}
\label{eq:diff}
\delta^{i}(u \otimes s)= \sum_{j=1}^{n} e_j u \otimes s x_j + \D u \otimes s
\end{equation}
for all $u\in A^{i}$ and $s\in S$. 
As before, the fact that this is a cochain complex is easily verified. 
Indeed, $\delta^{i+1}\delta^i (u\otimes s)$ equals
\begin{align*}
\sum_{k} e_k\bigg( \sum_{j} &
e_j u \otimes s x_j +\D{u} \otimes s\bigg) \otimes x_k + \D\bigg(\sum_{j} 
e_j u \otimes s x_j + \D{u} \otimes s\bigg) \\ 
&= \sum_{j,k} e_k e_j  u \otimes s x_j x_k +\sum_k e_k \D{u}\otimes s x_k
-\sum_j e_j \D{u}\otimes s x_j \\
&= 0,
\end{align*}
where we used the fact that $e_k e_j=-e_je_k$.  

\begin{rem}
The cochain complex \eqref{eq:univ aomoto} is independent of 
the choice of basis  $\{ e_1,\dots, e_n \}$ for $H^1(A)$.  Indeed, 
under the canonical identification $H^1(A)\otimes H_1(A)\cong 
\Hom (H^1(A),H^1(A))$, the element  $\sum_{j=1}^{n} e_j  \otimes  x_j$ 
used in defining the differentials $\delta^{i}$ corresponds to the 
identity map of $H^1(A)$.
\hfill $\Diamond$
\end{rem}

\begin{exm}
\label{ex:koszul}
Let $E=\bigwedge (e_1,\dots ,e_n)$ be the exterior algebra 
(with zero differential), and let $S=\C[x_1,\dots ,x_n]$ be its 
Koszul dual.  Then the cochain complex $(E^{\hdot}\otimes S,\delta)$ is 
simply the Koszul complex $K_{\hdot}(x_1,\dots,x_n)$. 
\hfill $\Diamond$
\end{exm}

More generally, if the $\cdga$ $A$ has zero differential, 
each boundary map $\delta^i \colon A^i\otimes S\to  A^{i+1}\otimes S$ 
is given by a matrix whose entries are linear forms in the 
variables $x_1,\dots ,x_n$.  In general, though, the entries of 
$\delta^i$ may also have non-zero constant terms, as can be 
seen in Examples \ref{ex:nonhomog}, \ref{ex:heis}, and \ref{ex:conf torus} 
below.

The relationship between the cochain complexes \eqref{eq:aomoto} 
and \eqref{eq:univ aomoto} is given by the following  lemma 
(for a more general statement, we refer to the proof of 
Lemma 8.8(1) from \cite{DP-ccm}).

\begin{lemma}
\label{lem:two aom}
The specialization of the cochain complex $A\otimes S$ 
at an element $a\in H^1(A)$ coincides with the cochain 
complex $(A,\delta_{a})$. 
\end{lemma}

\begin{proof}
\smartqed
Write $a=\sum_{j=1}^{n} a_j e_j\in H^1(A)$, and let 
$\m_a=(x_1-a_1,\dots , x_n-a_n)$ be the maximal ideal at $a$.  
The evaluation map $\ev_a \colon S\to S/\m_a=\C$ is the ring 
morphism given by $g\mapsto g(a_1,\dots, a_n)$.  The resulting 
cochain complex, $A (a)= A\otimes_S S/\m_a$, 
has differentials $\delta^i(a)$ given by
\begin{equation}
\label{eq:deltai}
\delta^i(a)(u)=\sum_{j=1}^{n} e_j u \otimes \ev_a(x_j) + \D{u}
= \sum_{j=1}^{n} e_j u\cdot  a_j +\D{u}= a\cdot u + \D{u}. 
\end{equation}
Thus, $A (a)=(A,\delta_{a})$, as claimed.
\qed
\end{proof}

In a completely analogous fashion, we may define a chain 
complex 
\begin{equation}
\label{eq:hom koszul}
\xymatrixcolsep{22pt}
(A_{\hdot} \otimes S,\partial) \colon 
\xymatrix{
\cdots \ar[r] 
&A_{i+1}\otimes S \ar^(.52){\partial_{i+1}}[r] 
&A_{i} \otimes S \ar^(.47){\partial_{i}}[r] 
&A_{i-1} \otimes S \ar[r] 
& \cdots}
\end{equation}
by essentially transposing the differentials of $(A^{\hdot} \otimes S,\delta)$.  
The previous lemma shows that the specialization of 
$(A_{\hdot} \otimes S,\partial)$ at an element $\alpha\in H_1(A)$ 
coincides with the chain complex $(A_{\hdot},\partial^{\alpha})$.

\subsection{Alternate views of resonance}
\label{subsec:supports}

As is well-known, the classical Koszul complex is exact.  
For an arbitrary $\cdga$, though, the cochain complex 
$(A^{\hdot}\otimes S,\delta)$ is not: its non-exactness 
is measured by the cohomology groups $H^i(A\otimes S)$, 
which are finitely generated modules over the polynomial ring $S$.   
This leads us to consider the support loci of these cohomology 
modules, 
\begin{equation}
\label{eq:alt res}
\wR^i(A) = \supp (H^i(A^{\hdot}\otimes S,\delta)),
\end{equation}
viewed again as algebraic subsets of the affine space $H^1(A)$. 

For instance, if $(E^{\hdot}\otimes S,\delta)=K_{\hdot}(x_1,\dots,x_n)$ 
is the Koszul complex from Example \ref{ex:koszul}, the support loci  
$\wR^i(E)$ vanish, for all $0\le i\le n$. 

We may also identify the polynomial ring $S$ with the 
symmetric algebra on $H^1(A)$, viewed as the coordinate 
ring of $H_1(A)$.  In this case, we have the support loci of 
the corresponding homology modules, 
\begin{equation}
\label{eq:alt hom res}
\wR_i(A) = \supp (H_i(A_{\hdot}\otimes S,\partial)), 
\end{equation}
which are algebraic subsets of the affine space $H_1(A)$. 
For a detailed discussion of support loci of chain complexes 
over an affine algebra, we refer to \cite{PS-mrl}.  

Since the ring $S$ is no longer a field  
(or even a PID, unless $H_1(A)=0$ or $\C$), the relation between these 
two types of support loci is not as straightforward as the one between 
the corresponding jump loci (see Example \ref{ex:nonhomog} below).  
Nevertheless, the cohomology jump loci and the homology support 
loci may be related in a filtered way, as follows.

\begin{theorem}
\label{thm:res compare}
For any finite-type $\cdga$ $(A,\D)$, and for any $q\ge 0$, 
the duality isomorphism $H^1(A)\cong H_1(A)$ restricts 
to an isomorphism 
\[
\bigcup_{i\le q} \RR^i(A) \cong \bigcup_{i\le q} \wR_i(A). 
\] 
\end{theorem}

\begin{proof}
\smartqed
As noted previously, the duality isomorphism $H^1(A)\cong H_1(A)$ 
identifies $\RR^i(A)$ with $\RR_i(A)$, for each $i\ge 0$.  

On the other hand, we know that $(A_{\hdot}\otimes S,\partial)$ is a chain 
complex of free, finitely generated modules over the affine $\C$-algebra $S$.  
Therefore, by Theorem 2.5 from \cite{PS-mrl}, we have that 
$\bigcup_{i\le q} \RR_i(A) = \bigcup_{i\le q} \wR_i(A)$, and the 
conclusion follows.
\qed
\end{proof}

As noted previously, when the $\cdga$ $A$ has differential $\D=0$,  
the boundary maps $\delta$ and $\partial$ from the chain complexes 
\eqref{eq:univ aomoto} and \eqref{eq:hom koszul} have entries 
which are linear forms in the variables of $S$.   
Consequently,  the sets $\RR^i(A)$ and $\wR_i(A)$ are homogeneous 
subvarieties of the affine space $A^1=H^1(A)$. 

\begin{corollary}
\label{cor:r1}
If $A$ has zero differential,  the resonance variety 
$\RR^1(A)\subset A^1$ is the vanishing locus of the 
codimension $1$ minors of the matrix of $S$-linear 
forms $\partial_2\colon A_2\otimes S\to A_1\otimes S$, 
or of its transpose, $\delta^1\colon A^1\otimes S\to A^2\otimes S$.
\end{corollary}

\begin{proof}
\smartqed
Using Theorem \ref{thm:res compare} and the above discussion, 
we obtain the equality $\RR^1(A)=\wR_1(A)$. By definition, $\wR_1(A)$ 
is the support locus of the $S$-module $H_1(A_{\hdot}\otimes S)=
\ker \partial_1/\im \partial_2$.  Writing $A_1=\C^n$ and $S=\C[x_1,\dots , x_n]$,  
we have that $\partial_1=\begin{pmatrix} x_1 & \cdots & x_n \end{pmatrix}$. 
The conclusion follows.
\qed
\end{proof}

We illustrate the theory with a simple, yet meaningful example, 
variants of which can also be found in \cite{MPPS, DPS-14}.

\begin{exm}
\label{ex:nonhomog}
Let $A$ be the exterior algebra on generators $a,b$ in 
degree $1$, endowed with the differential given by $\D{a}=0$ 
and $\D{b}=b\cdot a$. Then $H^1(A)=\C$, generated by $a$. 
Writing $S=\C[x]$, 
the chain complex \eqref{eq:hom koszul} takes the form 
\begin{equation}
\label{eq:toy}
\xymatrixcolsep{50pt}
A_{\hdot}\otimes S \colon 
\xymatrix{
S \ar^(.48){\partial_2=
\left(\begin{smallmatrix} 0 \\ x-1\end{smallmatrix}\right)
}[r] 
&S^2 \ar^(.5){\partial_1=
\left(\begin{smallmatrix} x & 0\end{smallmatrix}\right)
}[r] 
&S }.
\end{equation}
Hence, $H_1(A_{\hdot}\otimes S)=S/(x-1)$, and so $\wR_1(A)=\{1\}$.   
Using the above theorem, we conclude that $\RR^1(A)=\{0,1\}$. 

Note that  $\RR^1(A)$ is a non-homogeneous subvariety of $\C$.  
Note also that $H^1(A_{\hdot}\otimes S)=S/(x)$, and so $\wR^1(A)=\{0\}$, 
which differs from $\wR_1(A)$.

Finally, let $A'$ be the sub-$\cdga$ generated by $a$. Clearly, 
the inclusion map, $\iota\colon A'\inj A$, induces an isomorphism 
in cohomology. Nevertheless, $\RR^1(A')=\{0\}$, and so the resonance 
varieties of $A$ and $A'$ differ, although $A$ and $A'$ are quasi-isomorphic. 
\hfill $\Diamond$
\end{exm}

\begin{prb}
\label{prb:translated}
Can the resonance varieties of a $\cdga$ have positive-dimen\-sional 
irreducible components not passing through $0$?
\end{prb}

\section{The resonance varieties of a space}
\label{sect:models}

There are two basic types of resonance varieties that one 
can associate to a space, depending on which 
$\cdga$ is used to approximate it.  In this section, 
we discuss both types of resonance varieties, and 
several ways in which these varieties can be related.  

\subsection{The cohomology algebra}
\label{subsec:coho}

Throughout this section, $X$ will be a connected, finite-type CW-complex. 
The first approach (which has been in use since the 1990s) 
is to take the cohomology algebra 
$H^{\hdot}(X,\C)$, endowed with the zero differential, and let 
$\RR^i(X)$ be the resonance varieties of this $\cdga$. 
As indicated previously, these sets are homogeneous algebraic 
subvarieties of the affine space $H^1(X,\C)$.  

These varieties have been much studied in recent years, 
and have many practical applications, see e.g.~\cite{DPS-duke, 
PS-plms, PS-mrl, PS-crelle, Su-aspm, Su-pisa, Su-pau} 
and the references therein.  Let us just mention here two of their 
naturality properties. 

First, the resonance varieties are homotopy-type invariants.  More precisely, if 
$f\colon X\to Y$ is a homotopy equivalence, then the induced 
homomorphism $f^*\colon H^1(Y,\C)\to  H^1(X,\C)$ restricts to 
an isomorphism $f^* \colon \RR^i(Y)\isom \RR^i(X)$, for all $i\ge 0$, 
see e.g.~\cite{Su-aspm}.   

Next, if $p\colon Y\to X$ is a finite, regular cover, then the induced 
homomorphism, $p^*\colon H^1(X,\C)\inj H^1(Y,\C)$, maps 
each resonance variety $\RR^i(X)$ into $\RR^i(Y)$, 
with equality if the group of deck transformations 
acts trivially on $H^*(Y,\C)$, see e.g.~\cite{DP-pisa, Su-pau}.   

Yet the resonance varieties $\RR^i(X)$ 
do not always provide accurate enough information 
about the space $X$, since the cohomology algebra may not 
be a (rational homotopy) model for $X$.  It is thus 
important to look for alternate definitions of resonance 
in the non-formal setting.

\subsection{The Sullivan model}
\label{subsec:apl}

The second approach is to use Sullivan's model of polynomial forms, 
$\apl(X)$.  This is a rationally defined $\cdga$, whose construction 
is inspired by the de~Rham algebra of differential forms on a smooth 
manifold (see \cite{Su77}, \cite{FHT}). In particular, the cohomology algebra 
$H^{\hdot}(\apl(X))$ is isomorphic as a graded algebra to 
$H^{\hdot}(X,\C)$, via an isomorphism preserving $\Q$-structures.   
For a finite simplicial complex $K$, the model $\apl(K)$ admits a 
nice combinatorial description, closely related to the Stanley--Reisner 
ring of $K$ (see \cite{FJP}).

A connected, finite-type CW-complex $X$ is said to be {\em formal}\/ 
if its Sullivan model is formal, i.e., there is a weak equivalence 
$\apl(X) \simeq (H^{\hdot}(X,\C),0)$ preserving $\Q$-structures.  
The notion of $q$-formality of a space is defined analogously.  
Of course, if $X$ is formal, then it is $q$-formal, for all $q$. 
As a partial converse, if $X$ is $q$-formal and $\dim X \le q+1$, 
then $X$ is formal (see \cite{Ma}).  

Particularly interesting is the notion of $1$-formality.  It turns out that 
a space $X$ as above is $1$-formal if and only if its fundamental group, 
$\pi=\pi_1(X,x_0)$, is $1$-formal, that is, if the Malcev--Lie algebra of 
$\pi$ is the degree completion of a quadratic Lie algebra.  

For instance, if $H^*(X,\Q)$ is the quotient of a free $\cdga$ by 
an ideal generated by a regular sequence, then $X$ is a formal 
space (see \cite{Su77}). In particular, if $X$ has the rational cohomology 
of a torus, then $X$ is formal. For more on these formality notions, 
we refer to \cite{Ma, PS-imrn, DPS-duke, PS-formal, SW}. 

When $X$ is non-formal, the Sullivan model may have  
infinite-dimen\-sional graded pieces. In particular, the sets 
$\RR^i(\apl(X))$ are not {\it a priori}\/ algebraic sets.  
Thus, we will restrict our attention to spaces $X$ 
for which $\apl(X)$ can be replaced (up to weak equivalence) 
by a finite-type model $(A,\D)$.  

For this class of spaces, which 
includes many interesting examples of non-formal spaces, 
the resonance varieties $\RR^i(A)$  may be viewed as algebraic 
subsets of the affine space $H^1(X, \C)\cong  H^1(A)$. 

\subsection{An algebraic tangent cone theorem}
\label{subsec:alg tcone}

Before proceeding, let us briefly recall a standard notion in algebraic 
geometry.  Let $W\subset \C^n$ be a Zariski closed subset, 
defined by an ideal $I$ in the polynomial ring $S=\C[z_1,\dots ,z_n]$.   
The {\em tangent cone}\/ of $W$ at $0$ is the algebraic 
subset $\TC_0(W)\subset \C^n$ defined by the ideal 
$\init(I)\subset S$ generated by the initial forms of 
all non-zero elements from $I$.  This set is a homogeneous 
subvariety of $\C^n$, which depends only on the analytic 
germ of $W$ at zero.  In particular, 
$\TC_0(W)\ne \emptyset$ if and only if $0\in W$.  

In the previous two subsections, we associated two types 
of resonance varieties to a space $X$ having a finite-type 
model $A$.  The next  theorem, which may be viewed as 
an algebraic analogue of the Tangent Cone theorem, 
establishes a tight relationship between these two 
kinds of varieties.  

\begin{theorem}[\cite{MPPS}]
\label{thm:mpps}
Let $X$ be a finite-type CW-complex, and suppose there is a 
finite-type $\cdga$ $(A,\D)$ such that $\apl(X)\simeq A$. 
Then, for each $i\ge 0$, the tangent cone at $0$ to the 
resonance variety $\RR^i(A)$ is contained in $\RR^i(X)$.
\end{theorem}

As we shall see in Example \ref{ex:heis} below, the inclusion 
$\TC_0(\RR^i(A))\subseteq \RR^i(X)$ may well be strict.  

It seems natural to ask whether one can dispense in the above 
theorem with the hypothesis that the $\cdga$ $(A,\D)$ be realized 
by a space $X$, and distill a purely algebraic statement from it.

\begin{prb}
\label{prob:alg tcone}
Let $(A,\D)$ be a finite-type $\cdga$.  For each $i\ge 0$, determine 
whether the tangent cone at $0$ to $\RR^i(A)$ is contained 
in $\RR^i(H^{\hdot}(A))$. 
\end{prb}

\subsection{Positive weights}
\label{eq:poswt}
Under some additional hypothesis on the $\cdga$ under consideration, 
one can say more about the nature of its resonance varieties. 

Following Sullivan \cite{Su77} and Morgan \cite{Mo}, we say that 
a rationally defined $\cdga$ $(A,\D)$ has {\em positive weights}\/ 
if each graded piece can be decomposed into weighted pieces, 
with positive weights in degree $1$, and in a manner compatible 
with the $\cdga$ structure.  That is, 
\begin{enumerate}
\item For each $i\ge 0$, there is a vector space decomposition, 
$A^i=\bigoplus_{\alpha\in \Z} A^i_\alpha$.
\item 
$A^1_\alpha=0$, for all $\alpha\le 0$.
\item If $a\in A^i_\alpha$ and $b\in A^j_\beta$, 
then $ab\in A^{i+j}_{\alpha+\beta}$ and $\D{a}\in A^{i+1}_{\alpha}$. 
\end{enumerate}

A space $X$ is said to have positive weights if its Sullivan model does.  
If $X$ is formal, then $X$ does have positive weights: simply set the 
weight of a cohomology class in $A=H^{\hdot}(X,\C)$ equal to its degree.  
On the other hand, as we shall see in \S\S\ref{sect:qproj}--\ref{sect:elliptic}, 
the converse is far from true, even when $X$ is a smooth, complex 
algebraic variety. 

The existence of positive weights on a $\cdga$ model $A$ for 
$X$ imposes stringent conditions on the resonance varieties 
of $A$, and leads to an even tighter relationship between the 
resonance varieties of the space and its model. 

\begin{theorem}[\cite{DP-ccm, MPPS}]
\label{thm:mpps-bis}
Let $X$ be finite-type CW-complex, and suppose there is a 
rationally defined, finite-type $\cdga$ $(A,\D)$ with positive weights, 
and a $q$-equivalence between $\apl(X)$ and $A$ preserving $\Q$-structures.  
Then, for each $i\le q$, 
\begin{enumerate}
\item \label{r1} $\RR^i(A)$ is a finite union of rationally defined 
linear subspaces of $H^1(A)$.
\item \label{r2}  $\RR^i(A) \subseteq \RR^i(X)$.
\end{enumerate}
\end{theorem}

Once again, it seems natural to ask whether one can dispense 
with the hypothesis that $(A,\D)$ be a model for a finite-type 
CW-complex $X$.  

\begin{prb}
\label{quest:pos weights}
Let $(A,\D)$ be a finite-type $\cdga$ with positive weights. For each 
$i\ge 0$, determine whether $\RR^i(A)$ is contained in $\RR^i(H^{\hdot}(A))$,  
and whether $\RR^i(A)$ is a finite union of rationally defined linear subspaces.
\end{prb}

\begin{exm}
\label{ex:heis}
Let $X$ be the $3$-dimensional Heisenberg nilmanifold, i.e., 
the circle bundle over the torus, with Euler number $1$. Then 
$H^1(X,\C)=\C^2$, and all cup products of degree $1$ classes  
vanish; thus, $\RR^1(X)=H^1(X,\C)$.  

On the other hand, $X$ admits as a model $(A,\D)$ the exterior 
algebra on generators $a,b,c$ in degree $1$, 
with differential $\D{a}=\D{b}=0$ and $\D{c}=a\wedge b$.   
Clearly, this is a finite-dimensional model, with positive weights: 
simply assign weight $1$ to $a$ and $b$, and weight $2$ to $c$. 

Writing $S=\C[x,y]$, the chain complex \eqref{eq:hom koszul} 
takes the form 
\begin{equation}
\label{eq:heis}
\xymatrixcolsep{25pt}
A_{\hdot}\otimes S \colon 
\xymatrix{
\cdots \ar[r] & S^3 \ar^(.5){
\left(\begin{smallmatrix} 
y & 0 & 0\\[2pt]
-x & 0 & 0 \\[2pt]
1 & -x & -y \end{smallmatrix}\right)
}[rr] 
&&S^3 \ar^(.5){
\left(\begin{smallmatrix} x\, & y\, & 0\end{smallmatrix}\right)
}[rr] 
&&S }.
\end{equation}
It follows that $H_1(A_{\hdot}\otimes S)=S/(x,y)$, and so 
$\RR^1(A)=\{0\}$, a proper subset of $\RR^1(X)=\C^2$.
\hfill $\Diamond$
\end{exm}

\section{Characteristic varieties}
\label{sect:charvar}

We now turn to another type of homological jump loci associated to 
a space: the characteristic varieties, which keep track of jumps in 
the homology with coefficients in rank $1$ local systems. Closely 
related objects are the support loci for the Alexander modules. 

\subsection{Homology jump loci for  rank $1$ local systems}
\label{subsec:cv}

As before, let $X$ be a finite-type,  connected CW-complex.  Fix 
a basepoint $x_0$, and let $\pi=\pi_1(X,x_0)$ be its fundamental group. 
Finally, let $\Char(X)=\Hom(\pi,\C^*)$ 
be the algebraic group of complex-valued, multiplicative characters 
on $\pi$, with identity $1$ corresponding to the trivial representation. 
The identity component of this group, $\Char(X)^{0}$, is an algebraic 
torus of dimension $n=b_1(X)$; the other components are translates 
of this torus by characters corresponding to the torsion subgroup 
of $\pi_{\ab}=H_1(X,\Z)$. 

For each character $\rho\colon \pi\to \C^*$, let 
$\C_{\rho}$ be the corresponding rank $1$ local system on $X$.  
The {\em characteristic varieties}\/ of $X$ are the 
jump loci for homology with coefficients in such local systems, 
\begin{equation}
\label{eq:cvx}
\VV_i(X)= \{\rho \in \Char(X) \mid  H_i(X, \C_{\rho}) \ne 0\}.
\end{equation}

In more detail, let $X^{\ab}\to X$ be the 
maximal abelian cover, with group of deck transformations 
$\pi_{\ab}$.  Upon lifting the cell structure of $X$ 
to this cover, we obtain a chain complex of $\C[\pi_{\ab}]$-modules, 
\begin{equation}
\label{eq:equiv cc}
\xymatrixcolsep{20pt}
\xymatrix{\cdots \ar[r]& 
C_{i+1}(X^{\ab},\C) \ar^(.53){\partial^{\ab}_{i+1}}[r] & 
 C_{i}(X^{\ab},\C) \ar^(.45){\partial^{\ab}_{i}}[r] & 
  C_{i-1}(X^{\ab},\C) \ar[r] & \cdots  
}.
\end{equation}

Tensoring this chain complex with the $\C[\pi_{\ab}]$-module 
$\C_{\rho}$, we obtain a chain complex of $\C$-vector spaces,
\begin{equation}
\label{eq:eval cc}
\xymatrixcolsep{14pt}
\xymatrix{\cdots \ar[r]& C_{i+1}(X,\C_{\rho}) 
\ar^(.52){\partial^{\ab}_{i+1}(\rho)}[rr] &&
C_{i}(X,\C_{\rho}) \ar^(.47){\partial^{\ab}_{i}(\rho)}[rr] &&
C_{i-1}(X,\C_{\rho}) \ar[r] & \cdots  
},
\end{equation}
where the evaluation of $\partial^{\ab}_i$ at $\rho$ is obtained by applying 
the ring morphism $\C[\pi]\to \C$, $ g\mapsto \rho(g)$ to each entry.  
Taking homology in degree $i$ of this chain complex, we obtain the 
twisted homology groups $H_i(X, \C_{\rho})$, whose jumps in 
dimension the variety $\VV_i(X)$ keeps track of. 

In a similar fashion, we may define the cohomology jump loci 
$\VV^i(X)$ by the condition $H^i(X, \C_{\rho}) \ne 0$. 
Note that $H^i(X, \C_{\rho})\cong H_i(X, \C_{\rho^{-1}})$; thus, 
the inversion automorphism $\rho\mapsto \rho^{-1}$ of the 
character group of $X$ identifies $\VV^i(X)$ with $\VV_i(X)$, 
for each $i\ge 0$. 

\subsection{Some properties of the characteristic varieties}
\label{subsec:cv prop}

The sets $\VV_i(X)$ are algebraic subsets of the  character  
group $\Char(X)$.   Clearly, $1\in \VV_i(X)$ if and only if the 
$i$-th Betti number $b_i(X)$ is non-zero. 
In degree $0$, we have $\VV_0(X)= \{ 1 \}$. 
In degree $1$, the variety $\VV_1(X)$ depends only 
on the fundamental group $\pi=\pi_1(X,x_0)$---in fact, only on 
its maximal metabelian quotient, $\pi/\pi''$---%
so we shall sometimes denote it as $\VV_1(\pi)$.  

The characteristic varieties are homotopy-type invariants 
of our space.  More precisely, if $f\colon X\to Y$ is 
a homotopy equivalence, then the induced morphism between 
character group, $f^*\colon \Char(Y)\to \Char(X)$, restrict to 
an isomorphism $f^* \colon \VV_i(Y)\isom \VV_i(X)$, for all $i\ge 0$;  
see \cite{Su-imrn} for more details.

If $p\colon Y\to X$ is a finite, regular cover, then the induced 
morphism between character groups, $p^*\colon \Char(X)\inj \Char(Y)$, 
maps each characteristic variety $\VV_i(X)$ into $\VV_i(Y)$; 
see \cite{DP-pisa, Su-pau} for details.  

As noted in \cite{PS-plms}, the characteristic 
varieties behave well under finite 
direct products. More precisely, let $X_1$ and $X_2$ 
be two connected, finite-type CW-complexes. 
Identifying the character group of the product $X=X_1\times X_2$ with 
$\Char(X_1)\times \Char(X_2)$, we have 
\begin{equation}
\label{eq:cvprod}
\VV_i(X_1\times X_2)=\bigcup_{p+q=i}
\VV_p(X_1)\times \VV_q(X_2).
\end{equation}

The proof of this formula is straightforward:   For each 
character $\rho=(\rho_1,\rho_2)\in \Char(X)$, the chain complex   
$C_{\hdot}(X,\C_{\rho})$ decomposes as the tensor product 
of the chain complexes $C_{\hdot}(X_1,\C_{\rho_1})$ 
and $C_{\hdot}(X_2,\C_{\rho_2})$. 
Taking homology, we see that $H_i(X,\C_{\rho})=\bigoplus_{p+q=i} 
H_{p}(X_1,\C_{\rho_1}) \otimes_{\C} H_{q}(X_2,\C_{\rho_2})$, 
and the claim follows. 

\subsection{Alexander varieties}
\label{subsec:alexvars}

An alternative approach, going back to the definition 
of the Alexander polynomials of knots and links \cite{Al}, 
uses the homology modules of the universal 
abelian cover of our space $X$.   As before, 
let $\pi=\pi_1(X)$, and let 
\begin{equation}
\label{eq:alexinv}
H_i(X^{\ab},\C)=H_1(X, \C[\pi_{\ab}])
\end{equation} 
be the homology groups of the chain complex \eqref{eq:equiv cc}. 
These {\em Alexander invariants}\/ are in a natural way modules 
over the group ring $\C[\pi_{\ab}]$.  Identifying this commutative, 
Noetherian ring with the coordinate ring of the character group 
of $\pi$, we let 
\begin{equation}
\label{eq:alexvar}
\wV_i(X)= \supp ( H_i(X^{\ab},\C) ) 
\end{equation}
be the subvariety of $\Char(X)$ defined by the annihilator 
ideal of the respective homology module. 
One may also consider the cohomology modules $H^i(X^{\ab},\C)$ and 
their support varieties, $\wV^i(X)$; we will not pursue this  
approach here, but refer instead to \cite{LM} for details. 

As shown in \cite{PS-plms, PS-mrl}, the characteristic varieties  
and their homology support loci counterparts are 
related in the following way:
\begin{equation}
\label{eq:cval}
\bigcup_{i\le q} \VV_i(X)= \bigcup_{i\le q} \wV_i(X). 
\end{equation}

Of special interest is the first characteristic variety, 
$\VV_1(\pi)=\VV_1(X)$.   Suppose $\pi$ admits a 
finite presentation, say, 
$F/R=\langle x_1,\dots , x_n \mid r_1,\dots , r_m\rangle$, 
and let $\phi\colon F\surj \pi$ 
be the presenting homomorphism. 
Let $\partial_j r_i\in \Z[F]$ be the Fox derivatives of the relators,  
and let $\partial^{\ab}_2=(\partial_j r_i)^{\ab \circ \phi}$ be the 
corresponding Alexander matrix, with entries in $\Z[\pi_{\ab}]$.  
It follows from \eqref{eq:cval} that $\VV_1(\pi)$ 
coincides (at least away from $1$) with 
the zero locus of the ideal $E_1(\pi)$ of codimension $1$ 
minors of $\partial^{\ab}_2$, a result due to E.~Hironaka \cite{Hi97}. 

As shown in \cite{MS}, the first resonance variety $\RR_1(\pi)$ 
admits a similar description, at least when $\pi$ is a commutator-relators 
group, i.e., when all the relators $r_i$ belong to the commutator subgroup 
$[F,F]$. In this case, $\RR_1(\pi)$ is the zero locus of the codimension 
$1$ minors of the `linearized' Alexander matrix, 
$(\partial^{\ab}_2)^{\operatorname{lin}}$, which is 
the $m$ by $n$ matrix over the polynomial ring $S=\Z[y_1,\dots, y_n]$
with $ij$-entries equal to $\sum_{k=1}^{n}\epsilon(\partial_k\partial_j r_i) y_k$, 
where $\epsilon \colon \Z[F]\to \Z$ is the augmentation map.

\begin{rem}
\label{rem:univ}
The characteristic varieties can be arbitrarily complicated. 
For instance, let $f\in \Z[t_1^{\pm 1},\dots , t_n^{\pm 1}]$ be 
an integral Laurent polynomial. Then, as shown in \cite{SYZ}, there is a 
finitely presented group $\pi$ with $\pi_{\ab}=\Z^n$ and 
$\VV_1(\pi)=V(f)  \cup \{1\}$.  

More generally, let $Z$ be an algebraic subset of $(\C^*)^n$, defined 
over $\Z$, and let $k$ be a positive integer.  Then, as shown in 
\cite{Wa}, there is a finite, connected CW-complex $X$ with 
$\Char(X)=(\C^*)^n$ such that $\VV_i(X)=\{1\}$ for $i<q$ and 
$\VV_q(X)=Z  \cup \{1\}$. 
\hfill $\Diamond$
\end{rem}

\section{The tangent cone theorem}
\label{sect:tcone}

We are now ready to state a key result in the theory 
of cohomology jump loci: given a space $X$, and an algebraic 
model $A$ with good finiteness properties, the characteristic 
varieties $\VV^i(X)=\VV_i(X)$ may be identified around the identity with the 
resonance varieties $\RR^i(A)$. The resulting Tangent Cone 
theorem imposes strong restrictions on the nature of the resonance 
varieties $\RR^i(X)$ of a formal space $X$. 

\subsection{Two types of tangent cones}
\label{subsec:exp tc}

We start by reviewing two constructions which yield approximations  
to a subvariety $W$ of a complex algebraic torus $(\C^*)^n$. 
The first one is the classical tangent cone, a variant 
of the construction described in \S\ref{subsec:alg tcone}, 
while the second one is the exponential tangent cone, 
a construction first introduced in \cite{DPS-duke} and further 
studied in \cite{Su-imrn, SYZ}. 

Let $I$ be an ideal in the Laurent polynomial ring   
$\C[t_1^{\pm 1},\dots , t_n^{\pm 1}]$ such that $W=V(I)$.   
Picking a finite generating set for $I$, and multiplying 
these generators with suitable monomials if necessary, 
we see that $W$ may also be defined by the ideal $I\cap R$ 
in the polynomial ring $R=\C[t_1,\dots,t_n]$.  
Let $J$ be the ideal  in the polynomial ring 
$S=\C[z_1,\dots, z_n]$, generated by the polynomials 
$g(z_1,\dots, z_n)=f(z_1+1, \dots , z_n+1)$, 
for all $f\in I\cap R$. 

The {\em tangent cone}\/ of $W$ at $1$ is the algebraic 
subset $\TC_1(W)\subset \C^n$ defined by the ideal 
$\init(J)\subset S$ generated by the initial forms of 
all non-zero elements from $J$.  As before, the set 
$\TC_1(W)$ is a homogeneous subvariety of $\C^n$, 
which depends only on the analytic germ of $W$ at 
the identity.  In particular, $\TC_1(W)\ne \emptyset$ 
if and only if $1\in W$.  Moreover, $\TC_1$ commutes 
with finite unions.
  
On the other hand, the {\em exponential tangent cone}\/ 
to $W$ at the origin is the set  
\begin{equation}
\label{eq:tau1}
\tau_{1}(W)= \{ z\in \C^n \mid \exp(\lambda z)\in W,\ 
\text{for all $\lambda \in \C$} \}.
\end{equation}
As shown in \cite{DPS-duke, Su-imrn}, this set  
is a finite union of rationally defined linear subspaces of the affine 
space $\C^n$.  An alternative interpretation of this construction is 
given in \cite[\S6.3]{SYZ}.

It is readily seen that $\tau_1$ commutes with finite unions and 
arbitrary intersections. Clearly, the exponential tangent cone of 
$W$ only depends on the analytic germ of $W$ at the identity 
$1\in (\C^{*})^n$.  In particular, $\tau_1(W)\ne \emptyset$ 
if and only if $1\in W$. 

\begin{exm}
\label{ex:tau1 torus}
Suppose $W$ is an algebraic subtorus of $(\C^{*})^n$.  
Then  $\tau_1(W)$ equals $\TC_1(W)$, and both coincide with $T_1(W)$, 
the tangent space at the identity to the Lie group $W$. 
\hfill $\Diamond$
\end{exm}

More generally, there is always an inclusion between the two 
types of tangent cones associated to an  algebraic subset 
$W\subset (\C^{*})^n$, namely, 
\begin{equation}
\label{eq:ttinc}
\tau_{1}(W)\subseteq \TC_{1}(W). 
\end{equation}
But, as we shall see in several examples spread through this paper, 
this inclusion is far from being an equality for an arbitrary $W$. 
For instance, the tangent cone $\TC_1(W)$ may 
be a non-linear, irreducible subvariety of $\C^n$, or $\TC_1(W)$ 
may be a linear space containing the exponential tangent cone 
$\tau_{1}(W)$ as a union of proper linear subspaces.
   
\subsection{Germs of jump loci}
\label{subsec:germs}

As before, let $X$ be a connected, finite-type CW-complex.  Recall that, for 
each $i\ge 0$, we have a characteristic variety $\VV^i(X)$ inside 
the abelian, complex algebraic group $\Char(X)$.  Furthermore, 
the identity component of this algebraic group, $\Char(X)^{0}$, 
is isomorphic to $(\C^*)^n$, where $n=b_1(X)$.  

Now suppose we have a finite-type $\cdga$ model $(A,\D)$ 
for our space $X$.   Then, for each $i\ge 0$, we have a resonance variety 
$\RR^i(A)$ inside the affine space $H^1(A)=H^1(X,\C)$.  Furthermore,  
this affine space may be identified with $\C^n$, the tangent space at $1$ 
to $(\C^*)^n$.  

The next result, due to Dimca and Papadima \cite{DP-ccm}, relates the 
two types of cohomology jump loci around the origins of the respective 
ambient spaces.

\begin{theorem}[\cite{DP-ccm}]  
\label{thm:dp-ccm}
Suppose Sullivan's model $\apl(X)$ is $q$-equivalent to 
a finite-type $\cdga$ $(A,\D)$.  Then, for all $i\le q$, 
the germ at $1$ of $ \VV^i (X)$ is isomorphic to
the germ at $0$ of $ \RR^i (A)$. 
\end{theorem}

It is important to note that all the above isomorphisms are induced by an 
analytic isomorphism $\Char(X)^{0}_{(1)} \cong H^1(A)_{(0)}$, the inverse     
of which is obtained by suitably restricting the exponential map 
$\exp\colon \C^n \to (\C^*)^n$. 

Theorem \ref{thm:dp-ccm} shows that, at least around the origin, the 
resonance varieties of a finite-type $\cdga$ model for $X$ depend only on 
the characteristic varieties of $X$, and thus, only on the homotopy 
type of $X$. This observation leads to the following corollary.  

\begin{corollary}
\label{cor:resvar inv}
Let $X$ be finite-type CW-complex. Suppose $(A,\D)$ 
and $(A',\D')$ are two finite-type $\cdga$s, both $q$-equivalent 
to the Sullivan model $\apl(X)$, for some $q\ge 1$. There is 
then an isomorphism $H^1(A)\cong H^1(A')$ restricting to 
isomorphisms $\RR^i(A)_{(0)}\cong \RR^i(A')_{(0)}$, for all $i\le q$.  
\end{corollary}

A particular case of Theorem \ref{thm:dp-ccm} is worth singling out.  

\begin{corollary}
\label{cor:germ formal}
If $X$ is a $q$-formal space, then, for all $i\le q$, 
the germ at $1$ of $ \VV^i (X)$ is isomorphic to
the germ at $0$ of $ \RR^i (X)$. 
\end{corollary}

A precursor to this corollary can be found in the pioneering work 
of Green and Lazarsfeld \cite{GL87, GL91} on the cohomology 
jump loci of compact K\"{a}hler manifolds.   The case when 
$q=1$ was first established in \cite[Theorem A]{DPS-duke}.  For further 
developments in this direction, we refer to \cite{BW2, MPPS}.

\subsection{Tangent cones and jump loci}
\label{subsec:tcone thm}

Returning now to the general situation, consider  
an arbitrary connected, finite-type CW-complex $X$. We then 
have the following relationship (due to Libgober)  between the 
characteristic and resonance varieties of such a space. 

\begin{theorem}[\cite{Li02}]
\label{thm:lib}
For all $i\ge 0$, 
\begin{equation}
\label{eq:tc lib}
\TC_{1}(\VV^i(X))\subseteq \RR^i(X).
\end{equation}
\end{theorem}

Putting together these inclusions with those from \eqref{eq:ttinc}, 
we obtain the following immediate corollary. 
\begin{corollary}
\label{cor:tcone inc}
For all $i\ge 0$, 
\begin{equation}
\label{eq:tc inc}
\tau_{1}(\VV^i(X))\subseteq  \TC_{1}(\VV^i(X))\subseteq \RR^i(X).
\end{equation}
In particular, if $\RR^i(X)=\{0\}$, then $\tau_{1}(\VV^i(X))=\TC_{1}(\VV^i(X))=\{0\}$. 
\end{corollary}

In general, though, each of the inclusions from \eqref{eq:tc inc}, 
or both inclusions can be strict, as examples to follow will show. 

We now turn to spaces which admit finite-type algebraic models, 
and to the relations that hold between cohomology jump loci 
in this framework. 

\begin{theorem}
\label{thm:tc model}
Let $X$ be a connected, finite-type CW-complex, and suppose 
the Sullivan model $\apl(X)$ is $q$-equivalent to a finite-type $\cdga$ $A$.  
Then, for all $i\le q$, 
\begin{enumerate} 
\item \label{tcm1}
$\TC_{1}(\VV^i(X))= \TC_{0}( \RR^i(A) )$.
\item \label{tcm2}
If, moreover, $A$ has positive weights, and the $q$-equivalence 
between $\apl(X)$ and $A$ preserves $\Q$-structures, then 
$\TC_{1}(\VV^i(X))=  \RR^i(A)$.
\end{enumerate}
\end{theorem}

\begin{proof}
\smartqed
The first assertion follows at once from Theorem \ref{thm:dp-ccm}. 
The second assertion follows from the first one, when coupled with 
Theorem \ref{thm:mpps-bis}.
\qed
\end{proof}

The following examples show that the positive-weights assumption 
in Theorem \ref{thm:tc model}, part \ref{tcm2} is really necessary.  
That is, we cannot always replace 
$\TC_{0}( \RR^i(A) )$ with $\RR^i(A)$ in part \ref{tcm1}. 

\begin{exm}
\label{exm:solv}
Let $X=S^1$.    
We can take as a finite-dimensional model for the circle the 
$\cdga$ $(A,\D)$ from Example \ref{ex:nonhomog}, with 
$A=\bigwedge(a,b)$ and $\D{a}=0$, $\D{b}=b\cdot a$.  
Since $\VV^1(S^1)=\{1\}$, the resonance variety $\RR^1(A)=\{0,1\}$ properly contains 
$\TC_1(\VV^1(S^1))=\{0\}$. Of course, we can also take as a model for $S^1$ 
its cohomology algebra, $A'=\bigwedge(a)$, endowed with the zero differential, 
in which case the conclusion of part \ref{tcm2} is satisfied.
\hfill $\Diamond$
\end{exm}

\begin{exm}
\label{exm:solv}
Let $\varGamma$ be a discrete, co-compact subgroup of a simply-con\-nected,
solvable, real Lie group $G$, and let $M=G/\varGamma$ be the corresponding 
solvmanifold. 
As shown in \cite{Ka, PP}, all the characteristic varieties of 
$M$ are finite subsets of $\Char(M)$.  Moreover, as shown by 
Papadima and P\u{a}unescu in \cite{PP}, if $(A,\D)$ is any finite-dimensional model 
for $M$ (such as the one constructed by Kasuya \cite{Ka}), then all the resonance 
varieties $\RR^i(A)$ contain $0$ as an isolated point; in particular,
$\TC_{1}(\VV^i(M))= \TC_{0}( \RR^i(A) )=\{0\}$. 

Now suppose $G$ is completely solvable, and $(A,\D)$ is the classical 
Hattori model for the solvmanifold $M=G/\varGamma$.  Work of 
Millionschikov \cite{Mill}, as reprised in \cite{PP}, shows that $\RR^i(A)$ 
is also a finite set.  Furthermore, there are examples of solvmanifolds 
of this type where $\RR^i(A)$ is different from $\{0\}$. 
\hfill $\Diamond$
\end{exm}

\subsection{The influence of formality}
\label{subsec:tcone formal}

The main connection between the formality property 
of a space and its cohomology jump loci is provided by the 
following theorem.  (Again, the case $q=1$ was established  
in \cite{DPS-duke}, and the general case in \cite{DP-ccm}.)
 
\begin{theorem}[\cite{DPS-duke, DP-ccm}]
\label{thm:tcone}
If $X$ is a $q$-formal space, the following ``tangent cone formula" holds, 
for all $i\le q$, 
\begin{equation}
\label{eq:tc}
\tau_{1}(\VV^i(X))=\TC_{1}(\VV^i(X))=\RR^i(X).
\end{equation}
\end{theorem}

As an application of this theorem, we have the following characterization 
of the irreducible components of the cohomology jump loci in the formal setting.

\begin{corollary}
\label{cor:rational}
If $X$ is a $q$-formal space, then, for all $i\le q$, 
\begin{enumerate}
\item \label{tc1}
All irreducible components of the resonance variety 
$\RR^i(X)$ are rationally defined subspaces 
of $H^1(X,\C)$. 
\item \label{tc2}
All irreducible components of the characteristic variety 
$\VV^i(X)$ which pass through the origin are algebraic 
subtori of $\Char(X)^{0}$, of the form $\exp(L)$, where 
$L$ runs through the linear subspaces comprising $\RR^i(X)$.
\end{enumerate}
\end{corollary}

Even when the space $X$ is formal, the characteristic varieties $\VV^i(X)$ 
may have irreducible components which do not  pass through the identity 
of $\Char(X)^{0}$, and thus are not detected by the resonance varieties $\RR^i(X)$.  

\begin{exm}
\label{ex:knot}
Let $K$ be a non-trivial knot in the $3$-sphere, with 
complement $X=S^3 \setminus K$. Then $H^*(X,\Z)\cong H^*(S^1,\Z)$; 
therefore, $X$ is formal and $\RR^1(X)=\{0\}$. 
The characteristic variety $\VV^1(X)\subset \C^*$ 
consists of $1$, together with all the roots of the Alexander polynomial, 
$\varDelta_K\in \Z[t^{\pm 1}]$.  Thus, if $\varDelta_K\not\equiv 1$, 
then $\VV^1(X)$ has components which do not contain $1$. 
\hfill $\Diamond$
\end{exm}

\begin{exm}
\label{ex:mobius}
Let $X$ be the $2$-complex obtained by gluing a M\"{o}bius band 
to a $2$-torus along a meridian circle.  Then $X$ has the same rational 
cohomology ring as the $2$-torus; thus, $X$ is formal and $\RR^1(X)=\{0\}$.  
On the other hand, $\pi_1(X)=\langle x_1,x_2 \mid x_1x_2^2 = x_2^2 x_1\rangle$; 
hence, the variety $\VV^1(X)\subset (\C^*)^2$ consists of the identity together  
with the translated subtorus $t_1t_2^{-1}=-1$.
\hfill $\Diamond$
\end{exm}

\subsection{Formality tests}
\label{subset:discuss}

The next several examples illustrate the various ways in which 
the inclusions from Corollary \ref{cor:tcone inc} may fail to hold as equalities, 
thereby showing how Theorem \ref{thm:tcone} and Corollary \ref{cor:rational} 
can be used to detect non-formality. More examples will be given in 
Sections \ref{sect:qproj}--\ref{sect:elliptic}. 

\begin{exm}
\label{ex:toy}
Let $X$ be the presentation $2$-complex for the 
group $\pi=\langle x_1, x_2 \mid [x_1,[x_1, x_2]]\rangle$.   
In this case, $\VV^1(X)=\{t_1=1\}$, and so 
$\tau_{1}(\VV^1(X))=\TC_{1}(\VV^1(X))=\{x_1=0\}$.  
On the other hand, $\RR^1(X)=\C^2$, and so $X$ is not $1$-formal.
\hfill $\Diamond$
\end{exm}

The next example (adapted from \cite{DPS-duke}) 
shows how the rationality property from Corollary \ref{cor:rational} 
can be used as a formality test. 

\begin{exm}
\label{ex:sqr2}
Let $X$ be the presentation $2$-complex for the 
group $\pi$ with generators $x_1, \dots, x_4$ 
and relators $r_1=[x_1, x_2]$, $r_2=[x_1, x_4] [x_2^{-2}, x_3]$, 
and $r_3= [x_1^{-1}, x_3] [x_2, x_4]$.   Computing the 
linearized Alexander matrix of this presentation by 
the method described in \S\ref{subsec:alexvars},  
we see that $\RR^1(X)$ is the  quadric 
hypersurface in $\C^4$ given by the equation $z_1^2-2z_2^2=0$.  
This quadric splits into two linear subspaces defined 
over $\R$, but it is irreducible over $\Q$.  Thus, $X$ is not 
$1$-formal. 
\hfill $\Diamond$
\end{exm}

\begin{exm}
\label{example:ruled}
In view of Remark \ref{rem:univ}, there is a finitely presented 
group $\pi$ with abelianization $\Z^3$ and characteristic variety
\[
\VV^1(\pi)=\big\{ (t_1, t_2, t_3) \in (\C^*)^3 \mid (t_2-1)=(t_1+1)(t_3-1) \big\}.
\]
As noted in \cite{SYZ}, this variety is irreducible, and its exponential 
tangent cone at the origin splits as a union of two (rationally defined) 
lines in $\C^3$, 
\[
\tau_1(\VV^1(\pi))=\{x_2=x_3=0\}\cup \{x_1-x_3=x_2-2x_3=0\}.
\]
The variety $\VV^1(\pi)$ is a complex, $2$-dimensional torus 
passing through the origin. Nevertheless, this torus does not embed 
as an algebraic subgroup in $(\C^*)^3$; indeed, if it did, 
$\tau_1(\VV^1(\pi))$ would be a single plane.  Consequently, 
the group $\pi$ is not $1$-formal. 
\hfill $\Diamond$
\end{exm}

\section{Smooth quasi-projective varieties}
\label{sect:qproj}

We now switch our focus, from the general theory of cohomology 
jump loci to some of the applications of this theory within the class 
of smooth, complex quasi-projective varieties.  
For such spaces, the cohomology jump loci are severely 
restricted by the extra structure imposed on their $\cdga$ 
models and cohomology rings by the underlying algebraic geometry. 

\subsection{Compactifications and formality}
\label{subsec:qproj}

A complex projective variety is a subset of a complex projective 
space $\CP^n$, defined as the zero-locus of a homogeneous prime ideal in 
$\C[z_0,\dots, z_n]$.  A Zariski open subvariety of a projective 
variety is called a quasi-projective variety.  We will only consider here 
projective and quasi-projective varieties which are connected and smooth. 

If $M$ is a (compact, smooth) projective variety, then the 
Hodge decomposition on $H^*(M,\C)$ puts strong constraints 
on the topological properties of $M$.  For instance, 
as shown in \cite{DGMS}, every such a manifold is formal. 
  
Each smooth, quasi-projective variety $X$ admits a good compactification.  
That is to say, there is a smooth, complex projective variety $\overline{X}$  and a 
normal-crossings divisor $D$ such that $X=\overline{X}\setminus D$. 
By a well-known theorem of Deligne, each cohomology group of 
$X$ admits a mixed Hodge structure.   This additional structure puts 
definite constraints on the algebraic topology of such manifolds.   

For instance, if $X$ admits a smooth compactification 
$\overline{X}$ with $b_1(\overline{X})=0$, the weight 
$1$ filtration on $H^1(X,\C)$ vanishes; in turn, by work 
of Morgan \cite{Mo}, this implies the $1$-formality of $X$.   
Thus, as noted by Kohno in \cite{Ko}, if $X$ is the complement 
of a hypersurface in $\CP^n$, then $\pi_1(X)$ is $1$-formal.  
Moreover, if $n=2$, then $X$ itself is formal, see \cite{CM12, Ma}.

In general, though, smooth quasi-projective varieties need not 
be $1$-formal.  Furthermore, even when they are $1$-formal, 
they still can be non-formal. 

\begin{exm}
\label{ex:qp1f}
Let $E^{\times n}$ be the $n$-fold product of an elliptic curve.  The 
closed form $\frac{1}{2} \sqrt{-1} \sum_{i=1}^n dz_i \wedge d\bar{z}_i$ 
defines an integral cohomology class $\omega\in H^{1,1}(E^{\times n})$.  
By the Lefschetz theorem on $(1,1)$-classes, 
$\omega$ can be realized as the first Chern class of an algebraic 
line bundle over $E^{\times n}$.    Let $X_n$ be the complement 
of the zero-section of this bundle.  Then $X_n$ is a smooth, 
quasi-projective variety which is not formal.  In fact, as 
we shall see in Example \ref{ex:heis-bis}, $X_1$ 
is not $1$-formal.  On the other hand, $X_n$ is $1$-formal, for all $n>1$. 
\hfill $\Diamond$
\end{exm}

\subsection{Algebraic models}
\label{subsec:gysin}

As before, let $X$ be a connected, smooth quasi-projective variety, and  
choose a smooth compactification  $\overline{X}$ such that 
the complement is a finite union, $D=\bigcup_{j\in J} D_j$, 
of smooth divisors with normal crossings.
There is then a rationally defined $\cdga$, 
$A=A(\overline{X}, D)$, called the {\em Gysin model} 
of the compactification, constructed as follows.
As a $\C$-vector space, $A^i$ is the direct sum 
of all subspaces 
\begin{equation}
\label{eq:gysin}
A^{p,q}= \bigoplus_{\abs{S} =q} 
H^p \Big(\bigcap_{k\in S} D_k, \C\Big)(-q)
\end{equation}
with $p+q=i$, where $(-q)$ denotes the Tate twist. 
Furthermore, the multiplication in $A$ 
is induced by the cup-product in $\overline{X}$, and has the property that 
$A^{p,q} \cdot A^{p',q'} \subseteq A^{p+p',q+q'}$, while the differential, 
$\D \colon A^{p,q} \to A^{p+2, q-1}$, is constructed from the 
Gysin maps arising from intersections of divisors. 
The $\cdga$ just constructed depends on the compactification 
$\overline{X}$; for simplicity, though, we will denote it by $A(X)$ 
when the compactification is understood. 

An important particular case 
is when our variety $X$ has dimension $1$. That is to say, let $\varSigma$ 
be a connected, possibly non-compact, smooth algebraic curve.  
Then $\varSigma$ admits a canonical compactification,  
$\overline{\varSigma}$, and thus, a canonical Gysin model, 
$A (\varSigma)$. We illustrate the construction of this model 
in a simple situation, which we shall encounter again in Section \ref{sect:elliptic}. 

\begin{exm}
\label{ex:elliptic gysin}
Let $\varSigma=E^{*}$ be a once-punctured 
elliptic curve.  Then $\overline{\varSigma}=E$, and the Gysin 
model $A (\varSigma)$ is the algebra $A=\bigwedge(a,b,e)/(ae, be)$ 
on generators $a,b$ in bidegree $(1,0)$ and 
generator $e$ in bidegree $(0,1)$, with differential 
$\D\colon A\to A$ given by $\D{a}=\D{b}=0$ and $\D{e}=ab$. 
\hfill $\Diamond$
\end{exm}

The above construction is functorial, in the following sense: 
If $f\colon X\to Y$ is a morphism of quasi-projective manifolds 
which extends to a regular map $\bar{f}\colon \overline{X}\to \overline{Y}$ 
between the respective good compactifications, then there is an induced 
$\cdga$ morphism $f^{!}\colon A(Y)\to A(X)$ which respects the bigradings. 

Morgan showed in \cite{Mo} that the Sullivan model $\apl (X)$ 
is connected to to the Gysin model $A(X)$ by a chain  
of quasi-isomorphisms preserving $\Q$-structures. 
Moreover, setting the weight of $A^{p,q}$ equal to $p+2q$ defines a 
positive-weight decomposition on $(A^{\hdot}, d)$. 

In \cite{Du}, Dupont constructs a Gysin-type model for certain types 
of quasi-projective varieties, where the normal-crossing divisors 
assumption on the compactification can be relaxed. More precisely, 
let $\A$ be an arrangement of smooth hypersurfaces in a smooth, 
$n$-dimensional complex projective variety $\overline{X}$, 
and suppose $\A$ locally looks like an arrangement of 
hyperplanes in $\C^n$.  There is then a $\cdga$ model 
for the complement, $X=\overline{X}\setminus \bigcup_{L\in \A} L$, 
which builds on the combinatorial definition of the 
Orlik--Solomon algebra of a hyperplane arrangement 
(an algebra we will return to in \S\ref{subsec:arr stuff}).

\subsection{Configuration spaces}
\label{subsec:config}

In a special situation, an alternate model is available. 
A construction due to Fadell and Neuwirth associates 
to a space $X$ and a positive integer $n$ the space of 
ordered configurations of $n$ points in $X$, 
\begin{equation}
\label{eq:conf gamma}
\Conf(X,n) = \{ (x_1, \dots , x_n) \in X^{n} 
\mid x_i \ne x_j \text{ for } i\ne j\}.
\end{equation}

The $E_2$-term of the Leray spectral sequence for 
the inclusion $\Conf(X,n)\inj X^n$ was described concretely 
by Cohen and Taylor in the late 1970s. 
If $X$ is a smooth, complex projective variety 
of dimension $m$, then $\Conf(X,n)$ is a smooth, 
quasi-projective variety.  Moreover, as shown by Totaro in \cite{To96}, 
the Cohen--Taylor spectral sequence collapses at the 
$E_{m+1}$-term, and the $E_m$-term is a $\cdga$ model 
for the configuration space $\Conf(X,n)$.

The most basic example is the configuration space 
of $n$ ordered points in $\C$, which is a classifying 
space for $P_n$, the pure braid group on $n$ strings, 
whose cohomology ring was computed by Arnol'd in the 
late 1960s.  We shall come back to this example in 
\S\ref{sect:arr mf} in the setting of arrangements 
of hyperplanes, and we shall look at configuration space 
of $n$ points on an elliptic curve $E$ in \S\ref{sect:elliptic}, 
in the setting of elliptic arrangements.

More generally, following Eastwood and Huggett \cite{EH}, 
one may consider the ``graphic configuration spaces"  
\begin{equation}
\label{eq:conf}
\Conf(X, \varGamma)=  \{ (x_1, \dots , x_n) \in X^{\times n} 
\mid x_i \ne x_j \text{ for } \{i,j\}\in E(\varGamma)\}
\end{equation}
associated to a space $X$ and a simple graph $\varGamma$ with 
vertex set $[n]$ and edge set $E(\varGamma)$.  Especially interesting 
is the case when $X$ is a Riemann surface $\varSigma_g$.  For such 
a space, the naive compactification, $\overline{\Conf(X,n)}=X^{\times n}$, 
satisfies the hypothesis which permit the construction of the 
Dupont model, \cite{Du}.   For recent work exploiting this model, 
we refer to \cite{BMPP}. 

\subsection{Characteristic varieties}
\label{subsec:cv qproj} 

The structure of the jump loci for cohomology in rank $1$ local systems 
on smooth, complex projective and quasi-projective varieties (and, more 
generally, on K\"{a}hler and quasi-K\"{a}hler manifolds) was determined 
through the work of Beauville \cite{Be},  Green and 
Lazarsfeld \cite{GL87, GL91}, Simpson \cite{Sp92}, 
and Arapura \cite{Ar}.  

In the quasi-projective setting, further improvements 
and refinements were given in  \cite{Li01, ACM, Su-imrn}.  The definitive 
structural result was obtained by Budur and Wang in \cite{BW1}, 
building on work of Dimca and Papadima \cite{DP-ccm}. 

\begin{theorem}[\cite{BW1}]
\label{thm:bw}
Let $X$ be a smooth quasi-projective variety.  Then each 
characteristic variety $\VV^i(X)$ is a finite union of torsion-translated 
subtori of $\Char(X)$. 
\end{theorem}

Work of Arapura \cite{Ar} explains how the  
non-translated subtori occurring in the above decomposition of 
$\VV^1(X)$ arise.  Let us say that a holomorphic map $f \colon X \to \varSigma$ 
is {\em admissible}\/ if $f$ is surjective, has connected generic fiber, 
and the target $\varSigma$ is a connected, smooth complex curve 
with negative Euler characteristic.  Up to reparametri\-zation at the target, 
the variety $X$ admits only finitely many admissible maps; let 
$\mathcal{E}_X$ be the set of equivalence classes of such maps. 

If $f \colon X \to \varSigma$ is an admissible map, 
it is readily verified that $\VV^1(\varSigma)=\Char(\varSigma)$. 
Thus, the image of the induced morphism between character groups, 
$f^*\colon \Char(\varSigma)\to \Char(X)$, is an algebraic subtorus 
of  $\Char(X)$.

\begin{theorem}[\cite{Ar}]
\label{thm:arapura}
The correspondence $f \mapsto f^*(\Char(\varSigma))$ establishes 
a bijection between the set $\mathcal{E}_X$ of equivalence classes 
of admissible maps from $X$ to curves and the set of 
positive-dimensional, irreducible components of $\VV^1(X)$ 
containing $1$.
\end{theorem}

The positive-dimensional, irreducible components of $\VV^1(X)$ 
which do not pass through $1$ can be similarly described, by replacing 
the admissible maps with certain ``orbifold fibrations," whereby 
multiple fibers are allowed.  For more details and further explanations, 
we refer to \cite{ACM, Su-imrn}. 

\subsection{Resonance varieties}
\label{subsec:res qproj} 

We now turn to the resonance varieties associated to a 
quasi-projective manifold, and how they relate to the 
characteristic varieties.  The Tangent Cone theorem 
takes a very special form in this algebro-geometric setting. 

\begin{theorem}
\label{thm:tcone qp}
Let $X$ be a smooth, quasi-projective variety, and let $A(X)$ be a 
Gysin model for $X$.  Then, for each $i\ge 0$, 
\begin{equation}
\label{eq:tcone qp}
\tau_{1}(\VV^i(X)) =  \TC_{1}(\VV^i(X)) = \RR^i(A(X)) \subseteq  \RR^i(X).
\end{equation}
Moreover, if $X$ is $q$-formal, the last inclusion is an equality, 
for all $i\le q$.
\end{theorem}

\begin{proof}
\smartqed
By Theorem \ref{thm:bw}, each irreducible component of $\VV^i(X)$  
passing through $1$ is a complex algebraic subtorus $W\subset \Char(X)$.  
As noted in Example \ref{ex:tau1 torus}, $\tau_1(W)=\TC_1(W)$.  
Since both $\tau_1$ and $\TC_1$ commute with finite unions, 
the first equality in \eqref{eq:tcone qp} follows. 

Next, recall that $A(X)$ is a finite-type, rationally defined $\cdga$ 
which admits positive weights. Moreover, there is a weak equivalence 
between $A(X)$ and $\apl(X)$ preserving the respective $\Q$-structures.  
The second equality now follows from Theorem \ref{thm:tc model}, 
part \ref{tcm2}, while the last inclusion follows from 
Theorem \ref{thm:mpps-bis}.  
\qed
\end{proof}

In particular, the resonance varieties $\RR^i(A(X))$ are finite unions of 
rationally defined linear subspaces of $H^1(X,\C)$.  On the other hand, 
the varieties $\RR^i(X)$ can be much more complicated: for instance, 
they may have non-linear irreducible components. If $X$ is $q$-formal, 
though, Theorem \ref{eq:tcone qp} guarantees this cannot happen, as 
long as $i\le q$. 

\subsection{Resonance in degree $1$}
\label{subsec:res deg1} 

Once again, let $X$ be a smooth, quasi-projective variety, and 
let $A(X)$ be the Gysin model associated to a good compactification 
$\overline{X}$.  The degree $1$ resonance varieties $\RR^1(A(X))$, 
and, to some extent, $\RR^1(X)$, admit a much more precise description 
than those in higher degrees.  

As in the setup from Theorem \ref{thm:arapura}, let $\mathcal{E}_X$ 
be the set of equivalence classes of admissible maps from $X$ to 
curves, and let $f\colon X\to \varSigma$ be such map.  Recall 
from \S\ref{subsec:gysin} that the curve $\varSigma$ 
admits a canonical Gysin model, $A (\varSigma)$.  
As noted in  \cite{DP-ccm}, the induced 
$\cdga$ morphism, $f^{!} \colon A(\varSigma) \to A (X)$, 
is injective. Let $f^*\colon H^1(A(\varSigma)) \to H^1(A (X))$ 
be the induced homomorphism in cohomology.

\begin{theorem}[\cite{DP-ccm, MPPS}]
\label{thm:r1 gysin}
For a smooth, quasi-projective variety $X$, 
the decomposition of $\RR^1(A(X))$ into (linear) irreducible components 
is given by
\begin{equation}
\label{eq:pencils}
\RR^1(A(X))= \bigcup_{f\in \mathcal{E}_X} f^*(H^1(A(\varSigma))).
\end{equation}
\end{theorem}

If $X$ admits no admissible maps, i.e., $\mathcal{E}_X=\emptyset$, 
formula \eqref{eq:pencils} should be understood to mean $\RR^1(A(X))= \{0\}$ 
if $b_1(X)>0$ and $\RR^1(A(X))= \emptyset$ if $b_1(X)=0$.

\begin{exm}
\label{ex:heis-bis}
Let $X=X_1$  be the complex, smooth quasi-projective surface constructed in 
Example \ref{ex:qp1f}.  Clearly, this manifold is a $\C^*$-bundle over $E=S^1\times S^1$ 
which deform-retracts onto the Heisenberg manifold from Example \ref{ex:heis}.  
Hence, $\VV^1(X)=\{1\}$, and so 
$\tau_{1}(\VV^1(X))=\TC_{1}(\VV^1(X))=\{0\}$.  
On the other hand, $\RR^1(X)=\C^2$, and so $X$ is not $1$-formal.
\hfill $\Diamond$
\end{exm}

Under a $1$-formality assumption, the usual resonance varieties $\RR^1(X)$ 
admit a similar description. 

\begin{theorem}[\cite{DPS-duke}]
\label{thm:res kahler} 
Let $X$ be a smooth, quasi-projective variety, and suppose $X$ is $1$-formal. 
The decomposition into irreducible components of the first resonance 
variety is then given by
\begin{equation}
\label{eq:r1dec}
\RR^1(X)= \bigcup_{f\in \mathcal{E}_X} f^*(H^1(\varSigma,\C)), 
\end{equation}
with the same convention as before when $\mathcal{E}_X = \emptyset$.  
Moreover, all the (rationally defined) linear subspaces in this  
decomposition have dimension at least $2$, and any two 
distinct ones intersect only at $0$. 
\end{theorem}

If $X$ is compact, then the formality assumption 
in the above theorem is automatically satisfied, due to \cite{DGMS}.  
Furthermore, the conclusion of the theorem can also 
be sharpened in this case: each (non-trivial) irreducible component 
of $\RR^1(X)$ is even-dimensional, of dimension at least $4$.   

In general, though, the resonance varieties $\RR^1(X)$ can have 
non-linear components.  
For instance, if $X=\Conf(E,n)$ is the configuration space of $n\ge 3$ 
points on an elliptic curve $E$, then $\RR^1(X)$ is an irreducible, 
non-linear variety (in fact, a rational normal scroll), see \cite{DPS-duke} 
and also Example \ref{ex:conf torus} below. 

\section{Hyperplane arrangements and the Milnor fibration}
\label{sect:arr mf}

Next, we turn our focus to a class of quasi-projective varieties 
which are obtained by deleting finitely many hyperplanes from 
a complex affine space.  These hyperplane arrangement complements 
are formal spaces, yet the associated Milnor fibers may fail the 
Tangent Cone test for formality. 

\subsection{Complement and intersection lattice}
\label{subsec:arr stuff}

A hyperplane arrangement $\A$ is a finite collection 
of codimension $1$ linear subspaces in a complex affine 
space $\C^{n}$.   Its complement, 
$M(\A)=\C^{n}\setminus\bigcup_{H\in \A}H$, 
is a connected, smooth, quasi-projective variety. 
This manifold is a Stein domain, and thus has the 
homotopy-type of a finite CW-complex of dimension $n$. 
Moreover, $M(\A)\cong U(\A)\times \C^*$, where $U(\A)$ 
is the complement in $\CP^{n-1}$ of the projectivized arrangement. 

The topological invariants of the complement are intimately tied  
to the combinatorics of the arrangement.  The latter 
is encoded in the {\em intersection lattice}, $L(\A)$, which 
is the poset of all intersections of 
$\A$, ordered by reverse inclusion. The rank 
of the arrangement, denoted $\rk( \A )$, is the 
codimension of the intersection $\varSigma(\A)=\bigcap_{H\in \A} H$. 

\begin{exm}
\label{ex:br}
The braid arrangement of rank $n-1$ 
consists of the diagonal hyperplanes $H_{ij}=\{z_i-z_j=0\}$ 
in $\C^{n}$. The complement of this arrangement is the configuration 
space $\Conf(\C,n)$, while the intersection lattice is the lattice 
of partitions of $[n]=\set{1,\dots,n}$, ordered by refinement.  
\hfill $\Diamond$
\end{exm}

For each hyperplane $H\in \A$, pick a linear form $f_H\in \C[z_0,\dots ,z_n]$ 
such that  $\ker(f_H)=H$.   The homogeneous polynomial $Q(\A)=\prod_{H\in \A} f_H$, 
then, is a defining polynomial for the arrangement. 

Building on work of Arnol'd on the 
cohomology ring of $\Conf(\C,n)$,  Brieskorn 
showed in \cite{Br} that the closed $1$-forms ${\D{f}_H}/{f_H}$ 
generate the de~Rham cohomology of $M(\A)$.  Moreover,  
the inclusion of the subalgebra generated by those forms 
into  the de~Rham algebra $\varOmega^*_{\rm dR}(M(\A))$ 
induces an isomorphism in cohomology; consequently, 
the complement $M(\A)$ is a formal space.

In \cite{OS}, Orlik and Solomon gave a simple 
combinatorial description of the cohomology ring of the 
complement. Let $E=\bigwedge(\A)$ be the exterior algebra 
(over $\Z$) on degree-one classes $e_H$ 
dual to the meridians around the hyperplanes $H\in \A$, and set 
$e_\B=\prod_{H\in \B} e_H$ for each sub-arrangement $\B\subset \A$.  
Next, define a differential $\partial \colon E\to E$ 
of degree $-1$, starting from 
$\partial(e_H)=1$, and extending it to 
a linear map on $E$, using the graded Leibniz rule. 
Then
\begin{equation}
\label{eq:os}
H^*(M(\A),\Z)= \bigwedge(\A)/I(\A), 
\end{equation}
where $I(\A)$ is the (homogeneous) ideal generated by all 
elements of the form 
\begin{align}
\label{eq:osrels}
\hsp &e_{\B}, && \text{if $\varSigma(\B) =\emptyset$}, \hsp \\ \notag
\hsp &\partial e_{\B}, && \text{if $\codim \varSigma(\B)  < \abs{\B}$}. \hsp
\end{align}

More generally, suppose $\A$ is an arrangement of complex linear 
subspaces in $\C^n$.  Using a blow-up construction, De~Concini 
and Procesi gave in \cite{dCP95} a `wonderful' $\cdga$ model for 
the complement of such an arrangement. Based on a simplication 
of this model due to Yuzvinsky \cite{Yu}, Feichtner and Yuzvinsky 
showed in \cite{FeYu} the following:  If $L(\A)$ is a geometric lattice, 
then the complement of $\A$ is a formal space.  In general, 
however, the complement of a complex subspace arrangement 
need not be formal.  This phenomenon is illustrated in  
\cite{DeS07}, within the class of coordinate subspace arrangements, and 
in \cite{MW}, within the class of diagonal subspace arrangements. 

\subsection{Cohomology jump loci of the complement}
\label{subsec:cjl arr}

Once again, let $\A$ be a complex hyperplane arrangement.  
The resonance varieties of the arrangement, $\RR^i(\A):=\RR^i(M(\A))$, 
live inside the affine space $H^1(M(\A),\C)=\C^{\abs{\A}}$.   These 
varieties depend only on the Orlik--Solomon algebra of $\A$, and 
thus, only on the intersection lattice $L(\A)$.  

In \cite{Fa97}, Falk asked whether the resonance varieties $\RR^i(\A)$ 
are finite unions of linear subspaces.  A special case of the Tangent Cone 
theorem, proved in \cite{CS99} specifically for arrangement complements 
and in degree $i=1$, led to a positive answer to this question, at least 
for $\RR^1(\A)$.  With the technology provided by the general 
version of the Tangent Cone theorem, it is now easy to answer Falk's 
question in full generality.   Indeed, since the complement $M(\A)$ is 
a formal space, Corollary \ref{cor:rational} shows that $\RR^i(\A)$ is, 
in fact, a finite union of rationally defined linear subspaces, for each $i\ge 0$. 

In degree $i=1$, these linear spaces can be described much more precisely. 
Indeed, as shown by Falk and Yuzvinsky in \cite{FY} (see also \cite{LY, Mb}), 
each component of $\RR^1(\A)$ corresponds to a multinet on a 
sub-arrange\-ment of $\A$.  

Briefly, a {\em $k$-multinet}\/ on $\A$ is a partition into 
$k\ge 3$ subsets $\A_{\alpha}$, together with an assignment 
of multiplicities $m_H$ to each $H\in \A$, and a choice of rank $2$ 
flats, called the base locus.  All these data must satisfy certain 
compatibility conditions. For instance, any two hyperplanes from 
different parts of the partition intersect in the base locus, while the sum 
of the multiplicities over each part  is constant.  Furthermore, if 
$X$ is a flat in the base locus,  then the sum 
$n_{X}=\sum_{H\in\A_\alpha\cap \A_X} m_H$ is independent 
of $\alpha$.  The multinet is {\em reduced}\/ if all the 
$m_H$'s are equal to $1$.  If, moreover, all the 
$n_X$'s are equal to $1$, the multinet is, in fact, a 
{\em net}, a classical notion from combinatorial geometry.

Every $k$-multinet on $\A$ gives rise to an admissible map $M(\A)\to \varSigma$, 
where $\varSigma=\CP^1\setminus \{\text{$k$ points}\}$, and the converse 
also holds. Moreover, the set $\mathcal{E}_{M(\A)}$ of admissible maps 
(up to reparametrization at the target) from $M(\A)$ to curves coincides  
with the set of multinets (up to the natural $S_k$-permutation action on 
$k$-multinets) on subarrangements of $\A$, see \cite{FY, PS-beta}.  
The {\em essential}\/ components of $\RR^1(A)$ are those corresponding 
to multinets fully supported on $\A$.

\begin{exm}
\label{ex:braid arr}
Let $\A$ be a generic $3$-slice of the braid 
arrangement of rank $3$, with defining polynomial 
$Q(\A)=z_0z_1z_2(z_0-z_1)(z_0-z_2)(z_1-z_2)$. 
Take a generic plane section, and label the 
corresponding lines as $1$ to $6$. 
Then, the variety $\RR^{1}(\A)\subset \C^6$ has $4$ 
`local' components, corresponding to the triple 
points $124, 135, 236, 456$, and one essential  
component, corresponding to the $3$-net $(16| 25 | 34)$.  
\hfill $\Diamond$
\end{exm}

From Theorem \ref{thm:arapura}, we know that the characteristic 
varieties $\VV^i(\A):=\VV^i(M(\A))$ consists of subtori in $(\C^{*})^n$, possibly 
translated by roots of unity, together with a finite number of torsion points.   
By Theorem \ref{thm:tcone}, we have that 
$\TC_1(\VV^i(\A))=\RR^i(A)$. Thus, the components 
of $\VV^i(\A)$ passing through the origin are completely determined 
by $L(\A)$.  

As pointed out in \cite{Su02}, though, the characteristic variety 
$\VV^1(\A)$ may contain translated subtori, that is, components 
not passing through $1$.  It is still 
not known whether such components are combinatorially determined.

\subsection{The Milnor fibration}
\label{subsec:mf arr}

Once again, let $\A$ be a hyperplane arrangement in $\C^n$, 
with complement $M=M(\A)$ and defining polynomial 
$Q=Q(\A)$. As shown by Milnor \cite{Mi} in a more general context, 
the restriction of the polynomial map $Q\colon \C^n \to \C$ to the complement
is a smooth fibration, $Q\colon  M\to \C^*$. 

The typical fiber of this fibration, $Q^{-1}(1)$, is called 
the {\em Milnor fiber}\/ of the arrangement, and is denoted 
by $F=F(\A)$.  The Milnor fiber is a Stein domain of complex 
dimension $n$, and thus has the homotopy type of a finite 
CW-complex of dimension $n$. Furthermore, the monodromy 
homeomorphism $h\colon F\to F$ is given by $h(z)=\E^{2 \pi \I/m}z$, 
where $m=\abs{\A}$, and thus has order $m$.  

\begin{exm}
\label{ex:boolean}
The Boolean arrangement consists of the coordinate 
hyperplanes in $\C^n$; its complement is the 
complex algebraic torus $(\C^*)^{n}$.  The  
map $Q\colon (\C^*)^n \to \C^*$, 
$z\mapsto z_1\cdots z_{n}$ is 
a morphism of algebraic groups. Hence, the 
Milnor fiber $F= \ker Q$ is an algebraic subgroup, 
isomorphic to $(\C^*)^{n-1}$.
\hfill $\Diamond$
\end{exm} 

\begin{exm}
\label{ex:pencil}
Consider a pencil of $m$ lines in $\C^2$, 
with defining polynomial $Q=z_1^m-z_2^m$ and 
complement $M=\C^* \times \C\setminus \{\text{$m$ points}\}$. 
The Milnor fiber, then, is a smooth complex curve of genus 
$\binom{m-1}{2}$ with $m$ punctures. 
\hfill $\Diamond$
\end{exm} 

In general, though, the polynomial map $Q\colon \C^n\to \C$ will 
have a non-isolated singularity at $0$, and the topology of the 
Minor fiber $F=F(\A)$ will be much more difficult to ascertain. In particular, 
it is a long-standing open problem to decide whether the first Betti number $b_1(F)$ 
is determined by the intersection lattice of $\A$, and, if so, to find an explicit 
combinatorial formula for it,  see for instance \cite{CS95, DeS14, 
Di11, DP-pisa, Li12, Su-pau}.  Marked progress towards a 
positive solution to this problem was made recently in \cite{PS-beta}, 
using in an essential way the relationship between the varieties 
$\VV^1(\A)$ and $\RR^1(\A)$ provided by the Tangent Cone theorem, 
as well as the multinet interpretation of the components of $\RR^1(\A)$. 

To make this machinery work, one starts by viewing the 
Minor fiber $F$ as the regular, cyclic $m$-fold cover of the 
projectivized complement, $U=\PP(M)$, defined by 
the homomorphism $\delta\colon \pi_1(U)\to \Z_m$ which 
takes each meridian generator to $1$, see \cite{CS95} and also 
\cite{Su-conm, Su-pau}.  Embedding $\Z_m$ into $\C^*$ by 
sending $1\mapsto \E^{2 \pi \I/m}$, we may 
view $\delta$ as a character on $\pi_1(U)$. The relative 
position of this character with respect to the  
variety $\VV^1(U)\cong \VV^1(\A)$ determines the first Betti 
number of $F$, as well as the characteristic polynomial 
of the algebraic monodromy, $h_*\colon H_1(F,\C)\to H_1(F,\C)$.  

\subsection{Cohomology jump loci of the Milnor fiber}
\label{subsec:cvmf}

Very little is known about the homology with coefficients in 
rank $1$ local systems of the Milnor fiber of an 
arrangement $\A$.  Since $M(\A)$ is a smooth, quasi-projective 
variety, Theorem \ref{thm:bw}  
guarantees that the characteristic varieties $\VV^i(F(\A))$ 
are finite unions of torsion-translated subtori.  

Let $\pi\colon F(\A)\to U(\A)$ be the restriction of the Hopf fibration 
to the Milnor fiber.  Since $\pi$ is a finite, regular cover, we have that 
$\pi^*(\VV^i(U(\A)))\subseteq \VV^i(F(\A))$.  In general, though, this 
inclusion may well be strict. For instance, suppose $\A$ admits a 
non-trivial, reduced multinet, and let $T$ be the corresponding 
component of $\VV^1(\A)$.  Then, as shown in \cite{DP-pisa}, 
the variety $\VV^1(F(\A))$ has an irreducible component passing 
through the identity and containing $\pi^*(T)$ as a proper subset. 

\begin{exm}
\label{ex:cv mf braid}
Let $\A$ be the braid arrangement from Example \ref{ex:braid arr}.   
Recall that $\VV^1(\A)$ has four $2$-dimensional components, 
$T_1,\dots, T_4$, corresponding to the triple points, 
and also an essential, $2$-dimensional component $T$, corresponding 
to a $3$-net. The characteristic variety $\VV^1(F(\A))\subset (\C^*)^7$ 
has four $2$-dimensional components, $\pi^*(T_1), \dots ,\pi^*(T_4)$, 
as well as $4$-dimensional component $W$ which properly includes 
the $2$-torus $\pi^*(T)$.  
\hfill $\Diamond$
\end{exm}

Returning to the general situation, let again $\A$ be a complex hyperplane 
arrangement, and let $F=F(\A)$ be its Milnor fiber. By Theorem \ref{thm:tcone qp},  
$\TC_1(\VV^i(F))=\RR^i(A(F))$, where $A(F)$ is a Gysin model for $F$.  
Thus, to better understand the topology of the Milnor fiber, it would help  
a lot to address the following two problems. 

\begin{prb}
\label{prb:compactify mf}
Find a smooth compactification $\overline{F}$ 
such that $\overline{F}\setminus F$ is a normal-crossings divisor. 
Does the monodromy $h\colon F\to F$ extend to a diffeomorphism 
$\bar{h}\colon \overline{F}\to \overline{F}$?
\end{prb}

\begin{prb}
\label{prb:model mf}
Given a compactification $\overline{F}$ as above, 
write down an explicit presentation for the resulting 
Gysin model, $A(F)$.  Furthermore, compute the resonance varieties 
$\RR^i(A(F))$, and decide whether these varieties depend 
only on the intersection lattice $L(\A)$.
\end{prb}

\subsection{Formality of the Milnor fiber}
\label{subsec:zuber}

The following question was raised in \cite{PS-formal}, 
in a more general context:  Is the Milnor fiber  of a 
hyperplane arrangement $\A$ always formal?   
Of course, if $\rank(\A)=2$, then $F(\A)$ has the homotopy 
type of a wedge of circles, and so it is formal. 

If $\rank(\A)=3$, formality and $1$-formality are equivalent 
for the Milnor fiber, 
since in this case $F(\A)$ has the homotopy type of a $2$-complex. 
As noted by Dimca and Papadima \cite{DP-pisa}, 
if the monodromy map acts as the identity 
on $H_1(F(\A), \C)$, then $F(\A)$ is formal. 
In general, though, the Milnor fiber of an arrangement 
is not formal, as the following example of Zuber \cite{Zu} shows.  

\begin{exm}
\label{ex:zuber}
Let $\A$ be the arrangement associated to the complex 
reflection group $G(3,3,3)$, and defined by the polynomial 
$Q=(z_1^3-z_2^3)(z_1^3-z_3^3)(z_2^3-z_3^3)$. 
The resonance variety $\RR_1(\A)\subset \C^9$  has  
$12$ local components, corresponding to the triple points, 
and $4$ essential components corresponding to $3$-nets.

Consider the $3$-net whose associated rational map 
$\CP^2 \dashrightarrow \CP^1$ is given by 
$(z_1,z_2,z_3) \mapsto (z_1^3-z_2^3,z_2^3-z_3^3)$. 
This map restricts to an admissible map 
$U(\A)\to \CP^1 \setminus \{ (1,0), (0,1), (1,-1) \}$.  
Let $T$ be the essential, $2$-dimensional component 
of $\VV_1(U(\A))$ obtained by pullback along this pencil. 
Further pulling back $T$ via the covering projection 
$\pi\colon F(\A)\to U(\A)$ produces a $4$-dimensional subtorus 
inside $\Char(F(\A))=(\C^*)^{12}$.  

The subtorus $\pi^*(T)$ is of the form 
$\exp(L)$, for some linear subspace $L\subset H^1(F(\A),\C)$. 
Using the mixed Hodge structure on the cohomology of the 
Milnor fiber, Zuber showed in \cite{Zu} that $L$ 
cannot possibly be a component of the resonance variety $\RR^1(F(\A))$.    
Thus, the tangent cone formula from Theorem \ref{thm:tcone} is 
violated, and so the Milnor fiber $F(\A)$ is not $1$-formal. 
\hfill $\Diamond$
\end{exm}

In related work, Fern\'andez de Bobadilla gave in  \cite{FdB} an example 
of a quasi-homogeneous polynomial whose Milnor fibration has 
trivial geometric monodromy and whose Milnor fiber is simply-connected, 
yet non-formal. 

Zuber's example naturally leads to the following problem.  

\begin{prb}
\label{prb:mf tc}
Given a rank $3$ arrangement $\A$, decide whether the tangent 
cone formula holds for the Milnor fiber $F(\A)$.  Is this enough 
to guarantee that $F(\A)$ is formal? 
\end{prb}

\section{Elliptic arrangements}
\label{sect:elliptic}

We conclude with another class of arrangements, this time 
lying in a product of elliptic curves.   An especially convenient 
algebraic model is available for complements of 
`unimodular'  elliptic arrangements. Comparing the resonance 
varieties of this model to those of its cohomology algebra 
shows that complements of elliptic arrangements may be 
non-formal.

\subsection{Complements of elliptic arrangements}
\label{subsec:comp elliptic}

Let $E=\C/\Z^2$ be an elliptic curve.  We denote by 
$E^{\times n}$ be the $n$-fold product of such a curve.  
This is an abelian variety, with group law inherited from 
addition in $\C^n$.

An {\em elliptic arrangement}\/ in $E^{\times n}$ is a finite 
collection of fibers of group homomorphisms from $E^{\times n}$ to $E$.   
Each ``elliptic hyperplane" $H\subset E^{\times n}$ 
may be written as $H=f^{-1}(\zeta)$, for some point $\zeta\in E$ and 
some homomorphism $f\colon E^{\times n}\to E$ given by 
\begin{equation}
\label{eq:hom}
f(z_1,\ldots,z_n)=\sum_{j=1}^n c_{j} z_j, 
\end{equation}
where $c_j\in \Z$.  Thus, an 
arrangement $\A=\{H_1,\dots , H_m\}$ in $E^{\times n}$ 
is determined by an integral $m\times n$ matrix $C=(c_{ij})$ 
and a point $\zeta=(\zeta_1,\dots, \zeta_m)\in E^{\times m}$. 
We will write $\corank(\A):=n-\rank(C)$ and say that $\A$ is {\em essential}\/ 
if its corank is zero.

Let $L(\A)$ denote the collection of all connected components of
intersections of zero or more elliptic hyperplanes from $\A$.  
Then $L(\A)$ forms a finite poset under inclusion.  We say that 
$\A$ is {\em unimodular}\/ if all subspaces in $L(\A)$ are connected.  

Now let $M(\A)=E^{\times n} \setminus \bigcup_{H\in \A} H$ 
be the complement of our elliptic arrangement. This space  
is a smooth, quasi-projective variety.  Moreover, as shown in \cite{DSY}, 
the complement $M(\A)$ has the homotopy type of a CW-complex of 
dimension $n+r$, where $r=\corank(\A)$. Furthermore, if $r=0$, then 
$M(\A)$ is a Stein manifold. 

\subsection{An algebraic model}
\label{subsec:model elliptic}

Using the spectral sequence analyzed by Totaro in \cite{To96}, Bibby 
constructs in \cite{Bi} an algebraic model for the complement of a 
unimodular elliptic arrangement.  (An alternative approach is given by Dupont 
in \cite{Du}.)   Let us briefly review this construction, which 
generalizes the Gysin model of $E^*=E\setminus \{0\}$ 
described in Example \ref{ex:elliptic gysin}.

Let $a,b$ be the standard generators of $H^1(E,\Z)=\Z^2$.  
Applying the K\"{u}nneth formula, we may identify the cohomology ring 
$H^*(E^{\times n}, \Z)$ with the exterior algebra 
$\bigwedge (a_1,b_1, \dots ,a_n, b_n)$.   
For a homomorphism $f\colon E^{\times n}\to E$ as 
in \eqref{eq:hom}, the induced homomorphism in cohomology, 
$f^*\colon H^*(E,\Z)\to H^*(E^{\times n}, \Z)$, is given  by 
\begin{equation}
\label{eq:fstar}
f^*(a)=\sum_{j=1}^n c_j a_j, \quad f^*(b)=\sum_{j=1}^n  c_j b_j.
\end{equation}

Given an arrangement $\A=\{H_1,\dots , H_m\}$ in $E^{\times n}$, 
realize each elliptic hyperplane $H_i$ as a coset of the kernel of a homomorphism 
$f_i\colon E^{\times n}\to E$.  Next, consider the graded algebra 
\begin{equation}
\label{eq:ba}
A_{\Z}(\A)=\bigwedge\nolimits (a_1,b_1, \dots , 
a_n, b_n, e_1,\dots, e_{m})/I(\A), 
\end{equation}
where $I(\A)$ is the (homogeneous) ideal generated by the 
Orlik--Solomon relations \eqref{eq:osrels} among the 
generators $e_i$, together with the elements 
\begin{equation}
\label{eq:brels}
f_i^*(a) e_i, \ f_i^*(b) e_i, \qquad 1\le i\le m.
\end{equation}

Define a degree $1$ differential $\D$ on $A_{\Z}(\A)$ by setting 
$\D{a_i}=\D{b_i}=0$ and 
\begin{equation}
\label{eq:bdiff}
\D{e_i}=f_i^*(a) \wedge f_i^*(b) ,
\end{equation}
and extending $\D$ to the whole algebra by the graded Leibniz rule. 
Finally, let $A(\A)=A_{\Z}(\A)\otimes \C$, and extend $\D$ to $A(\A)$ 
in the obvious way.

\begin{theorem}[\cite{Bi}]
\label{thm:bibby}
Let $\A$ be a unimodular elliptic arrangement, and  
let $(A(\A),d)$ be the (rationally defined) $\cdga$ constructed above. 
There is then a weak equivalence $\apl(M(\A))\simeq A(\A)$ preserving 
$\Q$-structures.  
\end{theorem} 

In particular, we have an isomorphism 
$H^{\hdot}(M(\A),\C)\cong H^{\hdot}(A(\A),\D)$. 
Using this result, we obtain the following form of 
the tangent cone theorem for elliptic arrangements 
(the analogue of Theorem \ref{thm:tcone qp} in this context). 

\begin{theorem}
\label{thm:tcone elliptic}
Let $\A$ be a unimodular elliptic arrangement.  Then, for each $i\ge 0$, 
\begin{equation}
\label{eq:tcone ell}
\tau_{1}(\VV^i(M(\A))) =  \TC_{1}(\VV^i(M(\A))) = \RR^i(A(\A)) \subseteq  \RR^i(M(\A)), 
\end{equation}
with equality for $i\le q$ if $M(\A)$ is $q$-formal.   
\end{theorem}

\begin{proof}
\smartqed
The $\cdga$ model $(A(\A),\D)$ is finite-dimensional, since the underlying 
graded algebra $A(\A)$ is a quotient of a finitely-generated exterior algebra. 
Furthermore, this model has positive weights: simply assign 
weight $1$ to the generators $a_i, b_i$ and weight $2$ to the generators 
$e_i$.  

Using now Theorem \ref{thm:bibby}, the rest of the argument 
from Theorem \ref{thm:tcone qp} goes through, once we replace the Gysin 
model $A(M(\A))$ with the Bibby model $A(\A)$. 
\qed
\end{proof}

As a consequence, each resonance variety $\RR^i(A(\A))$ is a union of rationally 
defined linear subspaces.   As we shall see in \S\ref{subsec:ft}, that's not always  
true for the resonance variety $\RR^i(M(\A))$, in which case the last inclusion 
from \eqref{eq:tcone ell} fails to be an equality, and $M(\A)$ fails to be $i$-formal. 

It is worth noting that the Orlik--Solomon-type 
relations for the model $A(\A)$ are combinatorially determined, 
yet the relations \eqref{eq:brels} depend on the actual defining 
equations for the arrangement.  This observation leads to the 
following natural question.

\begin{prb}
\label{prob:combres}
Let $\A$ be a unimodular elliptic arrangement.  
Are the resonance varieties $\RR^i(A(\A))$ and 
$\RR^i(M(\A))$ determined by the intersection lattice of $\A$?  
Furthermore, is there a combinatorial criterion to decide whether 
the two varieties coincide, and, if so, whether 
the complement $M(\A)$ is formal?
\end{prb}

\subsection{Ordered configurations on an elliptic curve}
\label{subsec:ft}
The configuration space of $n$ points on an elliptic curve, 
$\Conf(E,n)$, is the complement of the elliptic braid arrangement, 
which is the arrangement in $E^{\times n}$ defined by the equations 
$z_i=z_j$ for $1\le i<j\le n$.  This space is a $K(\pi,1)$, 
with $\pi=PE_n$, the elliptic pure braid group on $n$ strings.

The resonance varieties $\RR^1(\Conf(E,n))$ were computed 
in \cite{DPS-duke}, while the positive-dimensional components 
of  $\VV^1(\Conf(E,n))$ were computed by Dimca in \cite{Di10}.  
An alternate way to perform this computation is to use work 
of Feler, namely, \cite[Theorem 3.1]{Fe}.  

Since $E=S^1\times S^1$ is a topological group, the  
space $\Conf(E,n)$ splits up to homeomorphism as a direct product, 
$\Conf(E^*, n-1)\times E$, where $E^*$ 
denotes the elliptic curve $E$ with the identity removed.  
Thus, for our purposes here it is enough to consider the configuration 
spaces $\Conf(E^*, n-1)$. In the next example, we work out in 
detail the case when $n=3$.  

\begin{exm}
\label{ex:conf torus}
Let $X=\Conf(E^*,2)$ be the configuration space of $2$ 
labeled points on a punctured elliptic curve.  This 
is the complement of the arrangement $\A$ in $E^{\times 2}$ 
defined by the polynomial $f=z_1z_2(z_1-z_2)$.  

By Theorem \ref{thm:bibby}, the space $X$ admits 
as a model the $\cdga$ $(A,\D)$, 
where $A$ is the exterior algebra on generators $a_1,b_1,a_2,b_2, 
e_1,e_2,e_3$ in degree $1$, modulo the ideal generated by 
the quadrics 
\[
a_1e_1,b_1e_1,a_2e_2,b_2e_2,(a_1-a_2)e_3, (b_1-b_2)e_3, 
(e_1-e_2)(e_1-e_3),
\]
while the differential $\D\colon A\to A$ is given by 
$\D{a}_1=\D{b}_1 =\D{a}_2=\D{b}_2=0$ and 
\[
\D{e}_1=a_1b_1, \  
\D{e}_2=a_2b_2, \
\D{e}_3=(a_1-a_2)(b_1-b_2).
\]

Identify $H^1(A)=\C^4$, with basis the classes represented by 
$a_1,b_1,a_2,b_2$, and let $S=\C[x_1,y_1,x_2,y_2]$ be the 
corresponding polynomial ring. Fixing bases as above 
for $A^1=\C^7$ and 
$\{a_1b_1, a_1a_2, a_1b_2,b_1a_2,b_1b_2,a_2b_2, 
a_2e_1, b_2e_1$, $a_1e_2, b_1e_2, a_1e_3, b_1e_3,
e_1e_2,e_1e_3\}$ 
for $A^2=\C^{14}$, we find that the boundary maps  
for the chain complex $A_{\hdot}\otimes S$ 
are given by 
\[
\setcounter{MaxMatrixCols}{14}
\arraycolsep=2pt
\partial_2=\begin{pmatrix}
-y_{1}&-x_{2}&-y_{2}&0&0&0&0&0&0&0&0&0&0&0\\
x_{1}&0&0&-x_{2}&-y_{2}&0&0&0&0&0&0&0&0&0\\
0&x_{1}&0&y_{1}&0&-y_{2}&0&0&0&0&0&0&0&0\\
0&0&x_{1}&0&y_{1}&x_{2}&0&0&0&0&0&0&0&0\\
1&0&0&0&0&0&{x_{2}}&{y_{2}}&0&0&0&0&0&0\\
0&0&0&0&0&1&0&0&{x_{1}}&{y_{1}}&0&0&0&0\\
1&0&{-1}&1&0&1&0&0&0&0&x_{1}+x_{2}&y_{1}+y_{2}&0&0
\end{pmatrix}
\]
and $\partial_1=\begin{pmatrix} x_1 &y_1& x_2 &y_2 &0& 0& 0\end{pmatrix}$. 
Computing homology, we find that $H_1(A_{\hdot}\otimes S)$ 
is presented by the $S$-linear map $\varphi\colon S^7\to S^3$ with matrix
\[
\arraycolsep=5pt
\varphi= \begin{pmatrix}
y_2& x_2& y_2& x_2& -y_2& -x_2\\
y_1& x_1& 0& 0& 0& 0\\
0& 0& y_2& x_2& y_1& x_1
\end{pmatrix}.
\]

By Theorem \ref{thm:res compare}, the resonance variety 
$\RR^1(A)$ is the zero locus of the ideal of $3\times 3$ minors 
of $\varphi$. An easy computation shows 
that this variety is the union of three planes in $\C^4$,  
\[
\RR^1(A)=\{x_1=y_1=0\}\cup \{x_2=y_2=0\}\cup \{x_1+x_2=y_1+y_2=0\}.
\]

On the other hand, the ring $H^{\hdot}(A)\cong H^{\hdot}(X,\C)$ 
is the exterior algebra on generators $a_1,a_2, b_1$, $b_2$ 
in degree $1$, modulo the ideal spanned by  
$a_1 b_2+a_2 b_1$, $a_1 b_1$, and $a_2 b_2$. 
Proceeding as above, we see that 
\[
H_1(H_{\hdot}(A)\otimes S) =\coker 
\arraycolsep=3pt
 \begin{pmatrix}
 y_2& x_2& -y_1& -x_1\\
y_1& x_1& 0& 0\\
0& 0& y_2& x_2
\end{pmatrix}.
\]
Hence, the first resonance variety of $X$ is an 
irreducible quadric hypersurface in $\C^4$, given by 
\[
\RR^1(X)=\{x_1y_2-x_2y_1=0\}.
\]
It follows from Corollary \ref{cor:rational} that the configuration 
space $X=\Conf(E^*,2)$ is not $1$-formal, a result already 
known from \cite{Be94}, \cite{DPS-duke}. 

Turning now to homology with coefficients in rank $1$ local systems, 
direct computation (recorded in \cite[Example 8.2]{Su-pisa}) shows that 
the first characteristic variety of $X$ consists of three 
$2$-dimensional algebraic tori inside $(\C^{*})^4$, 
\[
\VV^1(X)=\{ t_1=s_1=1\} \cup \{t_2=s_2=1\} \cup 
\{ t_1t_2=s_1s_2=1\}.  
\]

As noted in \cite[Proposition 5.1]{Di10}, these three subtori 
arise by pullback along the fibrations $\Conf(E^*,2) \to E^*$ 
obtained by sending a point $(z_1,z_2)$ to $z_2$, $z_1$, 
and $z^{}_1z_2^{-1}$, respectively.  Likewise, according to 
Theorem \ref{thm:r1 gysin}, the three planes 
comprising $\RR^1(A)$ are obtained by pulling back 
the linear space $H^1(A(E^*))=\C^2$ along the same fibrations. 
In particular,  
\[
\tau_1(\VV^1(X))=\TC_1(\VV^1(X))=\RR^1(A), 
\]
as predicted by Theorem \ref{thm:tcone elliptic}.

All three varieties are $2$-dimensional; thus, they are 
all properly contained in the $3$-dimensional 
variety $\RR^1(X)$. Therefore, the Tangent Cone theorem shows, 
once again, that $X$ is not $1$-formal. 
\hfill $\Diamond$
\end{exm}

\begin{acknowledgement}
I wish to thank Stefan Papadima for several useful conversations 
regarding this work, and also the referee, for a careful reading of 
the manuscript.
\end{acknowledgement}

\end{document}